\documentclass[12pt,a4paper]{article}
\usepackage{geometry}
\usepackage[table]{xcolor}

\usepackage{amsfonts}
\usepackage{amsmath}
\usepackage{amssymb}

\usepackage{amsthm}

\usepackage{graphicx}

\usepackage[hidelinks]{hyperref}
\usepackage{array}
\usepackage{tikz}
\usepackage{tikz-3dplot}
\usetikzlibrary{positioning}
\usetikzlibrary{arrows}
\usetikzlibrary{decorations.markings}
\usepackage{enumitem}

\usepackage{subcaption}

\newtheorem{prop}{Proposition}

\newtheorem{teo}{Theorem}

\theoremstyle{definition}
\newtheorem{defix}{Definition}
\newenvironment{defi}
  {\pushQED{\qed}\defix}
  {\popQED\enddefix}
\newtheorem{exex}{Example}
\newenvironment{exe}
  {\pushQED{\qed}\exex}
  {\popQED\endexex}

\newtheorem{obsx}{Remark}
\newenvironment{rem}
  {\pushQED{\qed}\obsx}
  {\popQED\endobsx}

  \title{Heteroclinic networks in coupled cell systems}
\author{Liliana Garrido da Silva$^{1}$ \and Pedro Soares$^{2\star}$\\
$^1$ Centro de Matem\'{a}tica da Universidade do Porto, Portugal \\
	$^2$ Instituto Superior de Economia e Gest\~{a}o, Universidade de Lisboa, Portugal\\
    $^{2\star}$psoares@iseg.ulisboa.pt}

\begin{document}
\maketitle

\abstract{A coupled cell system is an ODE system associated with a coupled cell network, where the dimension is determined by the number of cells. A heteroclinic connection is a set of solution trajectories between two equilibria of an ODE system. A realization of a heteroclinic network is an ODE system that exhibits equilibria corresponding to the nodes and heteroclinic connections between them according to the heteroclinic network. This paper investigates the realization of heteroclinic networks within coupled cell systems, focusing on embedding heteroclinic connections in 2D and 3D invariant subspaces. We adapt Field's method of embedding each heteroclinic connection in distinct 2D synchrony subspaces to support multiple connections within the same subspace. Using the concept of book embedding from graph theory, we demonstrate that any heteroclinic network can be realized using a coupled cell system with a number of cells proportional to the network's book-thickness. Additionally, we extend our analysis to 3D synchrony subspaces, allowing for more complex realizations. In this case, the number of cells necessary for a realization is proportional to the number of nodes in the heteroclinic 
network.}

%\keywords{Heteroclinic networks, Coupled cell systems, Synchrony subspaces, book embedding.}

%\msc{34C15; 34C37; 05C62.}

\section{Introduction}

Heteroclinic networks are a class of dynamical phenomena represented by directed graphs, where each node corresponds to an equilibrium point of a dynamical system, and each directed edge represents a trajectory connecting two equilibria.
Here, we consider dynamical systems given by a system of ordinary differential equations in $\mathbb{R}^n$.
Specifically, a heteroclinic connection is a trajectory that lies in the intersection of the unstable manifold of the source equilibrium and the stable manifold of the target equilibrium.
Usually, heteroclinic networks are defined as the union of heteroclinic cycles.
This is equivalent to assume that the directed graph is strongly connected, there is a directed path between any pair of nodes.

A central challenge in the study of heteroclinic networks is their realization: constructing a dynamical system whose phase portrait contains the desired network. 
The realization of a heteroclinic networks is not unique and different realizations of the same directed graph can exhibit different dynamical phenomena.
Realizations are particularly valuable when they are robust, meaning they persist under small perturbations within a given class of systems.
Robust realizations often rely on the presence of invariant subspaces, especially two-dimensional ones, as seen in equivariant systems \cite{CL23, ACL20}, coupled systems \cite{AP13, SM21}, and replicator dynamics \cite{PR22}.

In this work, we focus on realizing heteroclinic networks using coupled cell systems.
A coupled cell system is a dynamical system which respects the structure of a underlying coupled cell network, \cite{GST05, GS23}.
The coupled cell network is given by directed graph where the nodes and the directed arrows are called by cells and couplings, respectively.
The dimension of the coupled cell system is given by the number of cells, as each cell corresponds to a real variable of the ODE-system.
And, the dynamic on one variable depends on another if the is a coupling between the corresponding cells.
Coupled cell systems exhibit invariant spaces called synchrony subspaces, where subsets of cells evolve identically.
Synchrony subspaces are uniquely determined by the coupled cell network, via balanced colorings.

We consider homogeneous coupled cell networks with asymmetric inputs this means that all cells have the same type, but receive distinct couplings.
In this setting, the full-synchrony subspace, where all variables are equal, is one-dimensional and it serves as the location of equilibrium points in our constructions.
The heteroclinic connections are embedded in higher-dimensional synchrony subspaces, either two- or three-dimensional.

In this work, we are dealing with two different networks: heteroclinic networks and coupled cell networks.
To avoid ambiguity, we distinguish between heteroclinic networks (denoted with calligraphic letters) and coupled cell networks (denoted with standard capital letters). 
Moreover, we refer, respectively, to nodes and edges of a heteroclinic network $\mathcal{N}$ as equilibrium nodes and connections.
And, the nodes and edges of a coupled cell network $N$ are called cells and couplings, respectively.

\section{Results and discussion}

In this work, we explore realizations of heteroclinic network using coupling dynamical systems in two ways.
First, how many heteroclinic connections can be fitted in a 2D invariant subspace.
And second, how can we fit heteroclinic connections in a 3D invariant subspace.
Both ways lead to realizations with lower dimension than previous realizations presented in the literature.

Field \cite{F15} demonstrated that any heteroclinic network can be realized using a coupled cell system by embedding each heteroclinic connection in a distinct two-dimensional synchrony subspace. 
These subspaces can be visualized as "pages" in a book, with the full-synchrony subspace acting as the spine. 
Each page robustly supports a single heteroclinic connection between a saddle and a sink equilibrium.

In this work, we extend Field's construction by allowing multiple heteroclinic connections to share the same synchrony subspace. 
This leads us to adopt the concept of a book embedding from graph theory \cite{wiki,Y89}.
A book embedding is a 3D representation without self intersections of the graph with the nodes in a straight line and the edges on pages, half-planes with boundary in that straight line.
This concept is usually aplied to undirected graphs.
We adapt it to heteroclinic setting by imposing additional constraints: no node may have both incoming and outgoing connections on the same page, and no node may have multiple outgoing connections on a single page.
Borrowing the concept of book-thickness from graph theory as the minimum number of pages that a heteroclinic network needs to be book-embedded.
We prove that any heteroclinic network can be realized in a coupled cell system with a number of cells equal to the book-thickness of the network plus one.

We further generalize our approach by embedding heteroclinic connections in three-dimensional synchrony subspaces. In this setting, the unstable manifold at each node can have dimension two, allowing multiple outgoing connections from a single node within the same subspace.
We construct coupled cell networks with appropriate synchrony subspaces and show that any heteroclinic network with $n$ nodes can be realized in a system with $2n+1$ cells.
In this realization, each node is a saddle in one of these synchrony subspaces and a sink in the other synchrony subspaces.
This realization is almost complete, meaning that all but a measure-zero subset of the unstable manifolds belong to the stable manifolds.
In this approach, the dimension of the dynamical system realizing the heteroclinic network is proportional to the number of equilibria nodes.

%FALTA REFERENCIA A \cite{G20}

\section{Background and Preliminaries}

In this section, we review two different network concepts: coupled cell networks and heteroclinic networks.

\subsection{Coupled cell networks}

We follow the definition of coupled cell system given in \cite{GST05, GS23}.

Given a set of cells $C$ and a multi-set of couplings between the cells including self-loops $E\subset C\times C$.
The cells and couplings are classified into different types using equivalence relations on the set of cells $\sim_{C}$ and on the set of couplings $\sim_{E}$, respectively. 
Here, we assume that the cells have the same type, each cell receives exactly one coupling of each coupling type.
This is usually referred to as a homogeneous network with asymmetric inputs.
In this work, we use the term coupled cell network in this sense.
Numbering the edge types from $1$ to $k$, the multi-set of edges can be partitioned as $E=E_1\cup E_2\cup \dots\cup E_k$ where the edges in the set $E_i$ have type $i$ and $i=1,\dots,k$.

\begin{defi}
Given a set of cells $C$ and subsets $E_i\subset C\times C$, for $i=1,\dots,k$.
The tuple $(C,E_1,\dots,E_k)$ is a \emph{coupled cell network} if for each $1\leq i\leq k$ and each cell $c\in C$, there exists a unique $c_i\in C$ such that $(c_i,c)\in E_i$.
\end{defi}     

Along the manuscript, we denote the starting cell of the incoming coupling with type $i$ targeting cell $c\in C$ by $c_i$, $(c_i,c)\in E_i$, for $i=1,\dots,k$.

\begin{exe}
Let $C=\{0,1,2\}$, $E_1=\{(1,0),(0,1),(1,2)\}$ and $E_2=\{(2,0),(2,1),(0,2)\}$.
The triple $(C,E_1,E_2)$ is a coupled cell network with $3$ cells and $2$ edges types.
This network is the second network, from the left, in Figure~\ref{fig:pn}.
The first type of edge is represented by a normal arrow.
And the second type of edges is represented by a double arrowhead.
\end{exe}

In order to graphically represent different edge types, we use arrow with different heads and shapes, see Figure~\ref{fig:pn1}.

We associate a dynamical system to a given coupled cell network as follows.
Each cell $c$ corresponds to a variable $x_c\in\mathbb{R}$.
For a function $f:\mathbb{R}^{k+1}\rightarrow \mathbb{R}$, we associate the following admissible vector field $f^N:\mathbb{R}^n\rightarrow \mathbb{R}^n$ where $n$ are the number of cells in $N$ and 
$$ f^N_c(x)=f(x_c,x_{c_1},\dots,x_{c_k}).$$
The coupled cell system is given by the ODE-system
$$\dot{x}=f^N(x).$$
A perturbation of a coupled cell system refers to a perturbation of the function $f$.

\begin{exe}
Let $N$ be the second coupled cell network displayed in Figure~\ref{fig:pn}, from the left.  
The coupled cell systems associated with $N$ have the following form:
$$
\begin{cases}
\dot{x}_0 &= f(x_0, x_1, x_2) \\
\dot{x}_1 &= f(x_1, x_0, x_2) \\
\dot{x}_2 &= f(x_2, x_0, x_1)
\end{cases},
$$
for some function $ f: \mathbb{R}^3 \to \mathbb{R} $. 
\end{exe}

A key feature of  coupled cell systems is the existence of invariant subspaces given by the structure of the underlying network.
A polydiagonal is a subspace of $\mathbb{R}^n$ where certain cell coordinates are equal, i.e., $x_i=x_j$, for certain cells $i$ and $j$.
Given a polydiagonal $\Delta$, we can assign a coloring to the cells such that two cells share the same color if and only if their coordinates are equal in $\Delta$.
There exists a one-to-one correspondence between polydiagonals and cell colorings.
A cell coloring is a equivalence relation on the set of cells, we say that two cells share the same color if they belong to the same equivalence class.
Given a cell coloring $\bowtie$, we can define a corresponding polydiagonal $\Delta_{\bowtie}$ where coordinates are equal if the associated cells share the same color
$$\Delta_{\bowtie}=\{x\in\mathbb{R}:c\bowtie d \Rightarrow x_c=x_d\}.$$
We say that a coloring $ \bowtie $ is balanced if, for any two cells $c\bowtie d$, we have that $c_i\bowtie d_i$ for each edge type $i=1,\dots,k$.
A polydiagonal $\Delta_{\bowtie}$ is invariant under any admissible vector field if and only if the associated coloring $\bowtie$ is balanced.
Such invariant polydiagonals are referred to as synchrony subspaces.

For homogeneous networks, which is the case here, the full-synchrony subspace $\Delta_0=\{x\in\mathbb{R}^n:\forall_{c,d} x_c=x_d \}$ is always a synchrony subspace.
This subspace is one-dimensional and we will use it to place the equilibrium points.
Every synchrony subspace contains the full-synchrony subspace, and together they form a lattice structure.
A minimal synchrony subspace is one that contains no other synchrony subspace except the full-synchrony subspace. 
Heteroclinic connections will be embedded within these minimal synchrony subspaces.
Note that the intersection of any two minimal synchrony subspaces is precisely the full-synchrony subspace.

\begin{exe}
Let $N$ be the second coupled cell network displayed in Figure~\ref{fig:pn}, from the left.
The full-synchrony subspace $\Delta_0=\{x_0=x_1=x_2\}$ is invariant for any coupled cell system associated with $N$.
Moreover, this network admits two synchrony subspaces with dimension 2:
$$\Delta_1 = \{ x_0 = x_2 \}, \quad \Delta_2 = \{ x_0 = x_1 \}.$$
The corresponding balanced coloring are given by the classes $\{\{0,2\},1\}$ and $\{\{0,1\},2\}$.

The coloring $\{0,\{1,2\}\}$ is not balanced.
The cells $1$ and $2$ have the same color and they receive an edge of the first type from the cells $0$ and $1$.
However, the cells $0$ and $1$ do not share the same color.
So, the coloring is not balanced and the polydiagonal $\{ x_1 = x_2 \}$ is not invariant for some coupled cell systems associated with $N$.
\end{exe}

\subsection{Heteroclinic networks}

Consider a smooth dynamical system defined on $ \mathbb{R}^n $ governed by
$$\dot{x} = F(x), \quad x \in \mathbb{R}^n,$$
where $ F: \mathbb{R}^n \to \mathbb{R}^n $ is a smooth vector field. 
Take an equilibrium point $\xi \in \mathbb{R}^n$, $F(\xi) = 0$. 
The stable and unstable manifolds of $\xi$ are denoted as $ W^s(\xi) $ and $ W^u(\xi) $ and  are defined as:
$$W^s(\xi) = \{ x \in \mathbb{R}^n : \lim_{t \to +\infty} \phi_t(x) = \xi \}, \quad
W^u(\xi) = \{ x \in \mathbb{R}^n : \lim_{t \to -\infty} \phi_t(x) = \xi \},$$
where $\phi_t(x)$ denotes the flow of the dynamical system at time $ t $ with initial condition given by $x$.
Let $\xi$ and $\xi'$ be two equilibria of the dynamical system. 
A heteroclinic connection from $\xi$ to $\xi'$ is the image of a trajectory $\phi_t(x)$ such that:
$$
\lim_{t \to -\infty} \phi_t(x) = \xi_i, \quad \lim_{t \to \infty} \phi_t(x) = \xi_j.
$$
The existence of such trajectory implies that $ W^u(\xi) \cap W^s(\xi') \neq \emptyset$.
If $\xi=\xi'$, we call the connection homoclinic.

A heteroclinic cycle is a finite collection of equilibria $ \{ \xi_1, \xi_2, \dots, \xi_k \} $ and heteroclinic connections $ \gamma_i $ such that:
$$
\lim_{t \to -\infty} \gamma_i(t) = \xi_i, \quad \lim_{t \to \infty} \gamma_i(t) = \xi_{i+1}, \quad \text{with } \xi_{k+1} := \xi_1.
$$
A heteroclinic network is a union of such heteroclinic cycles. 
We represent a heteroclinic network using a directed graph $\mathcal{N}$, where nodes correspond to equilibria and edges represent heteroclinic connections. 
A dynamical system is said to realize the heteroclinic network $\mathcal{N}$ if its phase portrait contains equilibria and heteroclinic connections as described by $\mathcal{N}$.

A realization is called robust if it persists under small perturbations of the system.
In particular, we will look to perturbations within the class of coupled cell systems associated with a fixed network structure $N$.

A realization of a heteroclinic network may not be ``visible'' if the unstable manifolds of the equilibrium points are not sufficiently contained within the heteroclinic network. 
A realization is said to be:
\begin{itemize}
\item Complete if the unstable manifolds of every equilibrium are entirely contained in the heteroclinic network.
\item Almost complete if the unstable manifolds of every equilibrium are contained in the heteroclinic network up to a set with zero Lebesgue measure, \cite[Definition 2.6]{ACL20}.
\end{itemize}

In order to avoid the confusing between the two networks structures, heteroclinic networks are named using calligraphic capital letter $\mathcal{N}$ and coupled cell networks are labeled using printed capital letters $N$.

\section{Realization via book embedding}\label{sec:realbookemb}

In this section, we adapt the realization of heteroclinic networks using coupled cell systems presented in \cite{F15}.
Field's realization of a heteroclinic network \cite{F15} places the equilibrium points in the full-synchrony line and each heteroclinic connection belongs to a different 2D synchrony subspaces.
This can be thought as a book where each page corresponds to a different $2D$ synchrony subspace and the equilibria are along the spine of the book.  
In graph theory, the book embedding of a graph is $3D$ representation of the graph where nodes are in the spine (line), each edge lies entirely within one page (plane) and the edges do not intersection.
book embedding allow each page to contain more than one edge.
Given a book embedding of a heteroclinic network, we show that there is a robust realization using coupled cell systems with as many cells as the number of pages, plus one.
We prove this result in the next subsection by adjusting Field's realization to allow for multiple heteroclinic connections in the same 2D synchrony subspaces.
In the last subsection, we illustrate some example and relate this realization with others from the literature.

First, we adapt the concept of book embedding from graph theory to the context of heteroclinic networks. And state the main result.

\begin{defi}
Let $G=(V,E)$ be a directed graph. 
A book embedding consists of:
\begin{itemize}
	\item a straight-line $l\subset\mathbb{R}^3$,
	\item a set of planes/pages $H_1,\dots,H_k\subset \mathbb{R}^3$,
	\item an embedding of the graph in $\mathbb{R}^3$, $\phi:G\rightarrow \mathbb{R}^3$,
\end{itemize}
 such that 
\begin{itemize}
	\item each vertex $v$ corresponds to a distinct point in the straight-line $p\in l$, $\phi(v)=p\in l$,
	\item each edge $e= v_1\rightarrow v_2$ corresponds to a smooth curve from $\phi(v_1)$ to $\phi(v_2)$ in one of the planes $H_i$, $\phi(e)\subset H_i$,
	\item the planes intersect exactly at the straight-line $l$, $H_i\cap H_j=l$,
	\item the curves only intersect the straight-line $l$ at the start and end point, $\phi(e)\cap l= \{\phi(v_1),\phi(v_2)\}$
	\item any pair of edges do not intersect, outside the straight-line $l$, $[\phi(e_1)\cap \phi(e_2)]\setminus l=\emptyset$
\end{itemize}
and
\begin{itemize}
	\item in each plane, each vertex has exclusively incoming edges or outgoing edges,
	\item in each plane, each vertex has at most two outgoing edges, one in each half-plane.
\end{itemize}

The book thickness of a graph is the minimal number of pages of any book embedding of that graph.
\end{defi}

The last two restrictions are not part of the usually book embedding from graph theory.

\begin{exe}\label{ex:hetnet8}
In Figure~\ref{fig:hetnet8}, the heteroclinic network on the left can be book embedded as shown on the right.
The embedding uses two pages, $H_1$ and $H_2$, each containing two edges.
The top equilibrium node has both incoming and outgoing connections, but these are placed on different pages, as required by our constraints.
Any book embedding of this heteroclinic network has two pages.
So, the book thickness of this network is equal to $2$.

Although, this network can be embedded in a single page under standard graph-theoretic definitions, without the last two constraints.
In particular, if some node has both incoming and outgoing connections on the same page, our dynamical realization of the heteroclinic network would be impossible.
The intersection of the invariant planes used in our realization of the heteroclinic network is an invariant line.
And the connection approaches or escapes the line depending if the connection is incoming or outgoing.
As the connections do not belong to the invariant line, a 2D invariant plane can not have trajectories approaching and escaping the same equilibrium point.
\end{exe}

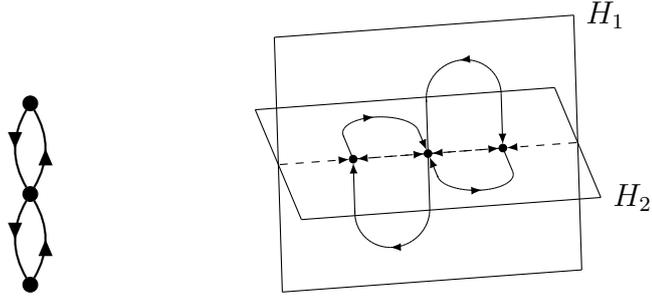
\begin{figure}[h]
\centering
%\begin{subfigure}{.10\textwidth}
\hspace{10mm}
\begin{tikzpicture}[
    % Style for arrows in the middle
    midarrow/.style={
        postaction={decorate},
        decoration={markings, mark=at position 0.5 with {\arrow[scale=0.8,fill=black]{triangle 45}}}
    }
]
\node (n1) [draw,shape=circle,fill=black,scale=0.5]   {};
\node (n2) [draw,shape=circle,fill=black,scale=0.5] [below=of n1]  {};
\node (n3) [draw,shape=circle,fill=black,scale=0.5] [below=of n2]  {};

\draw[midarrow, thick] (n1) to [bend right] (n2);
\draw[midarrow, thick] (n2) to [bend right] (n1);
\draw[midarrow, thick] (n2) to [bend right] (n3);
\draw[midarrow, thick] (n3) to [bend right] (n2);
\end{tikzpicture}
%\end{subfigure}
%\begin{subfigure}{.30\textwidth}
\hspace{20mm}
\tdplotsetmaincoords{115}{70}
\begin{tikzpicture} [tdplot_main_coords,scale=0.5]

\tdplotsetrotatedcoords{60}{90}{90}
\node[tdplot_rotated_coords, inner sep=0, outer sep=0] (O) at (0,0,0) {};
\node (P1) [tdplot_rotated_coords,circle,draw,fill=black, inner sep=0pt, outer sep =0pt, minimum size=1mm] at (0,0,2) {};
\node (P2) [tdplot_rotated_coords,circle,draw,fill=black, inner sep=0pt, outer sep =0pt, minimum size=1mm] at (0,0,4) {};
\node (P3) [tdplot_rotated_coords,circle,draw,fill=black, inner sep=0pt, outer sep =0pt, minimum size=1mm] at (0,0,6) {};
\node[tdplot_rotated_coords, inner sep=0, outer sep=0] (A) at (0,0,8) {};

%eixo diagonal
\draw[-latex, dashed] (O)  to (P1);
\draw[-latex, dashed] (P1) to (P2);
\draw[-latex, dashed] (P2) to (P1);
\draw[-latex, dashed] (P2) to (P3);
\draw[-latex, dashed] (P3) to (P2);
\draw[-latex, dashed] (A) to (P3);

%palnos H1
\tdplotsetrotatedcoords{60}{90}{190}
\node (E2) [tdplot_rotated_coords,inner sep=0pt,outer sep =0pt,label=right:$H_1$] at (3.5,0,8) {};
\draw[ tdplot_rotated_coords] (-3.5,0,8)--(3.5,0,8);
\draw[ tdplot_rotated_coords] (3.5,0,0)--(3.5,0,8);
\draw[ tdplot_rotated_coords] (-3.5,0,0)--(3.5,0,0);
\draw[ tdplot_rotated_coords] (-3.5,0,0)--(-3.5,0,8);

%P2->P3
\draw[ tdplot_rotated_coords] (P2) to (1.5,0,4);
%\draw[dashed, tdplot_rotated_coords] (1.5,0,4) to (3,0,4);
\foreach \t in {0,5,...,180}{
\draw[ tdplot_rotated_coords] ({1.5 + sin(\t)},{0},{5 + cos(\t)}) -- ({1.5 + sin(\t + 5)},{0},{5 + cos(\t + 5)});}
\draw[-latex, tdplot_rotated_coords] ({1.5 + sin(90)},{0},{5 + cos(90)}) -- ({1.5 + sin(90 + 5)},{0},{5 + cos(90 + 5)});
%\draw[-latex,  tdplot_rotated_coords,dashed] (3,0,6) to  (1.5,0,6);
\draw[-latex,  tdplot_rotated_coords] (1.5,0,6) to  (P3);

%P2->P1
\draw[ tdplot_rotated_coords] (P2) to (-1.5,0,4);
\foreach \t in {0,5,...,180}{
\draw[ tdplot_rotated_coords] ({-1.5 - sin(\t)},{0},{3 + cos(\t)}) -- ({-1.5 - sin(\t + 5)},{0},{3 + cos(\t + 5)});}
\draw[-latex, tdplot_rotated_coords] ({-1.5 - sin(90)},{0},{3 + cos(90)}) -- ({-1.5 - sin(90 + 5)},{0},{3 + cos(90 + 5)});
\draw[-latex,  tdplot_rotated_coords] (-1.5,0,2) to  (P1);

%plano H2
\tdplotsetrotatedcoords{60}{90}{90}
\node (E2) [tdplot_rotated_coords,inner sep=0pt,outer sep =0pt,label=right:$H_2$] at (3.5,0,8) {};
\draw[ tdplot_rotated_coords] (-3.5,0,8)--(3.5,0,8);
\draw[ tdplot_rotated_coords] (3.5,0,0)--(3.5,0,8);
\draw[ tdplot_rotated_coords] (-3.5,0,0)--(3.5,0,0);
\draw[ tdplot_rotated_coords] (-3.5,0,0)--(-3.5,0,8);

%P3->P2
\draw[latex-, tdplot_rotated_coords] (P2) to (1.5,0,4);
%\draw[dashed, tdplot_rotated_coords] (1.5,0,4) to (3,0,4);
\foreach \t in {0,5,...,180}{
\draw[ tdplot_rotated_coords] ({1.5 + sin(\t)},{0},{5 + cos(\t)}) -- ({1.5 + sin(\t + 5)},{0},{5 + cos(\t + 5)});}
\draw[latex-, tdplot_rotated_coords] ({1.5 + sin(90)},{0},{5 + cos(90)}) -- ({1.5 + sin(90 + 5)},{0},{5 + cos(90 + 5)});
%\draw[-latex,  tdplot_rotated_coords,dashed] (3,0,6) to  (1.5,0,6);
\draw[  tdplot_rotated_coords] (1.5,0,6) to  (P3);

%P1->P2
\draw[latex-, tdplot_rotated_coords] (P2) to (-1.5,0,4);
\foreach \t in {0,5,...,180}{
\draw[ tdplot_rotated_coords] ({-1.5 - sin(\t)},{0},{3 + cos(\t)}) -- ({-1.5 - sin(\t + 5)},{0},{3 + cos(\t + 5)});}
\draw[latex-, tdplot_rotated_coords] ({-1.5 - sin(90)},{0},{3 + cos(90)}) -- ({-1.5 - sin(90 + 5)},{0},{3 + cos(90 + 5)});
\draw[  tdplot_rotated_coords] (-1.5,0,2) to  (P1);

\end{tikzpicture}
%\end{subfigure}

\caption{Heteroclinic network(left) and a book embedding of this heteroclinic network(right), where each connection is embedded in a distinct 2D plane (page).}
\label{fig:hetnet8}
\end{figure}

%
%In order to robustly realize the heteroclinic network as sketched by a book embedding, the heteroclinic connections will be on an invariant plane and avoid heteroclinic connections along the invariant line.
%Suppose that a page has one incoming connection and one outgoing connection from the same equilibrium point, violating the last two restriction. 
%The equilibrium point must be a saddle node in the corresponding invariant plane.
%Then one of the heteroclinic connection must be on the invariant line which we intend to avoid.

A simple book embedding of any heteroclinic network (without homoclinic connections) is given by placing each edge in a different page.
Field's \cite{F15} showed that every heteroclinic network (without homoclinic connections) can be realized using coupled cell systems as sketched in these simple book embeddings.
We extend this result for any book embedding as defined above.

\begin{teo}\label{teo:bookembedding}
Let $\mathcal{N}$ be a heteroclinic network without homoclinic connections and let $B$ be a book embedding of $\mathcal{N}$ with $k$ pages. Then the heteroclinic network can be robustly realized in a coupled cell system with $k+1$ cells.

In particular, any heteroclinic network with book-thickness $b$ and no homoclinic connections can be robustly realized in a coupled cell system with $b+1$ cells.
\end{teo}

This theorem shows that the dimension where a heteroclinic network is realized using a coupled cell system is directly tied to how efficiently the network can be embedded in a book format. 
The fewer pages needed, the fewer cells required in the coupled cell network and the lower is the dimension of the dynamical system.

\subsection{Proof}

Next, we reproduce the realization presented in \cite{F15} and adjust it to any book embedding.

{\bf Construction of the coupled cell network $P_n$}

First, we describe how to construct a family of coupled cell networks that have the desired 2D synchrony subspaces.

The coupled cell network considered in the realization is inductively defined as follows. 
Let $P_1$ be the first network in Figure~\ref{fig:pn} with two cells and one type of edge.
Let $P_n$ be the network with cells $\{0,1,\dots, n\}$ and edges $E_1, E_2,\dots, E_n$ divided by types.
The cells of network $P_{n+1}$ are $\{0,1,\dots, n,n+1\}$ and there are $n+1$ edge types.
For $j=1,\dots,n$, the edges of type are $E_j\cup\{(j,n+1)\}$.
And the edges of new type, $n+1$, are  $\{(n+1,c):c=0,1\dots,n\}\cup \{(0,n+1)\}$.
Figure~\ref{fig:pn} displays the networks $P_1$, $P_2$, $P_3$ and $P_5$ obtained from the previous construction.

\begin{figure}[h]
\begin{subfigure}{.1\textwidth}
\begin{tikzpicture}
\node (n1) [circle,draw]   {0};
\node (n2) [circle,draw] [below=of n1]  {1};

\draw[->, thick] (n2) to (n1);
\draw[->, thick] (n1) to (n2);
\end{tikzpicture}
\end{subfigure}
\begin{subfigure}{.20\textwidth}
\begin{tikzpicture}
\node (n1) [circle,draw]   {0};
\node (n2) [circle,draw] [below=of n1]  {1};
\node (n3) [circle,draw]  [right=of n1] {2};

\draw[->, thick] (n2) to (n1);
\draw[->, thick] (n1) to (n2);
\draw[->, thick] (n2) to (n3);

\draw[->>, thick] (n3) to (n1);
\draw[->>, thick] (n3) to  (n2);
\draw[->>, thick] (n1) to  (n3);
\end{tikzpicture}
\end{subfigure}
\begin{subfigure}{.20\textwidth}
\begin{tikzpicture}
\node (n1) [circle,draw]   {0};
\node (n2) [circle,draw] [below=of n1]  {1};
\node (n3) [circle,draw]  [right=of n1] {2};
\node (n4) [circle,draw] [below=of n3]  {3};

\draw[->, thick] (n2) to (n1);
\draw[->, thick] (n1) to (n2);
\draw[->, thick] (n2) to (n3);
\draw[->, thick] (n2) to (n4);

\draw[->>, thick] (n3) to (n1);
\draw[->>, thick] (n3) to  (n2);
\draw[->>, thick] (n1) to  (n3);
\draw[->>, thick] (n3) to  (n4);

\draw[->>>, thick] (n4) to (n1);
\draw[->>>, thick] (n4) to  (n2);
\draw[->>>, thick] (n4) to  (n3);
\draw[->>>, thick] (n1) to  (n4);
\end{tikzpicture}
\end{subfigure}
\begin{subfigure}{.4\textwidth}
\begin{tikzpicture}
\node (n1) [circle,draw]   {0};
\node (n2) [circle,draw] [right=of n1]  {1};
\node (n3) [circle,draw]  [below right=of n2] {2};
\node (n4) [circle,draw]  [below left=of n3] {3};
\node (n5) [circle,draw] [left=of n4]  {4};
\node (n6) [circle,draw]  [above left=of n5] {5};

\draw[->, thick] (n2) to (n1);
\draw[->, thick] (n1) to (n2);
\draw[->, thick] (n2) to (n3);
\draw[->, thick] (n2) to (n4);
\draw[->, thick] (n2) to (n5);
\draw[->, thick] (n2) to (n6);

\draw[->>, thick] (n3) to (n1);
\draw[->>, thick] (n3) to (n2);
\draw[->>, thick] (n1) to (n3);
\draw[->>, thick] (n3) to (n4);
\draw[->>, thick] (n3) to (n5);
\draw[->>, thick] (n3) to (n6);

\draw[->>>, thick] (n4) to (n1);
\draw[->>>, thick] (n4) to (n2);
\draw[->>>, thick] (n4) to (n3);
\draw[->>>, thick] (n1) to (n4);
\draw[->>>, thick] (n4) to (n5);
\draw[->>>, thick] (n4) to (n6);

\draw[->>>>, thick] (n5) to (n1);
\draw[->>>>, thick] (n5) to (n2);
\draw[->>>>, thick] (n5) to (n3);
\draw[->>>>, thick] (n5) to (n4);
\draw[->>>>, thick] (n1) to (n5);
\draw[->>>>, thick] (n5) to (n6);

\draw[->>>>>, thick] (n6) to (n1);
\draw[->>>>>, thick] (n6) to (n2);
\draw[->>>>>, thick] (n6) to (n3);
\draw[->>>>>, thick] (n6) to (n4);
\draw[->>>>>, thick] (n6) to (n5);
\draw[->>>>>, thick] (n1) to (n6);
\end{tikzpicture}
\end{subfigure}
\caption{Examples of coupled cell networks $P_n$ for $n=1,2,3,5$ used to realize heteroclinic network with connections in 2D synchrony subspaces.}
\label{fig:pn}
\end{figure}
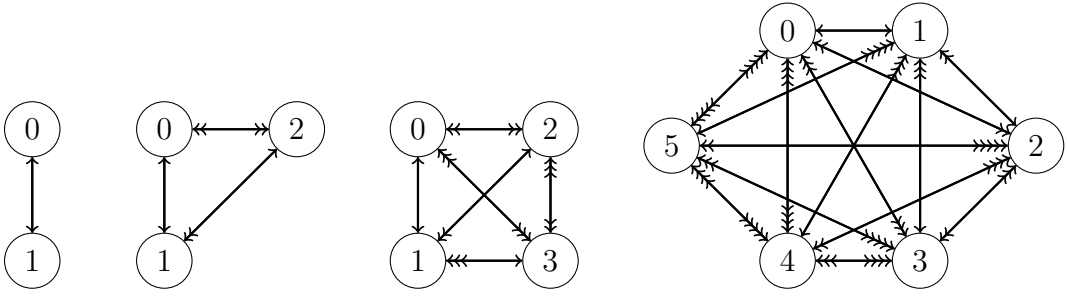

Following \cite[Proposition 3.3 and Lemma 4.5]{F15}, the networks $P_n$ have $n+1$ cells and $n$ asymmetric inputs. 
Moreover, for every $j=1,\dots,n$, $\Delta_j=\{x_0=x_i, i\neq j\}$ is a 2D-synchrony subspace of $P_n$. 
In each synchrony subspace, $\Delta_j$, the coupled cell systems has the following form:
\begin{equation}\label{eq:ccs2d}
\begin{cases}
\dot{x_0}=f(x_0,x_0,\dots,x_0,\overbrace{x_j}^{\text{input j}},x_0,\dots,x_0)\\
\dot{x_j}=f(x_j,x_0,\dots,x_0,x_0,x_0,\dots,x_0)\\
\end{cases},
\end{equation}
for some $f:\mathbb{R}^{n+1}\rightarrow \mathbb{R}$.
Let $p\in \Delta_0\subset \mathbb{R}^{n+1}$ be an full synchronous equilibrium of the coupled cell system $\dot{x}=f^{P_n}(x)$. 
The eigenvalues of the Jacobian matrix of (\ref{eq:ccs2d}) at $p$ are $f_0(p)+f_1(p)+\dots+f_{n}(p)$ and $f_0(p)-f_j(p)$, where $f_i$ is the first derivative of $f(y_0,y_1,\dots,y_n)$ with respect to $y_i$. 
Moreover, the eigenvalues of the coupled cell system $\dot{x}=f^{P_n}(x)$ at a full synchronous equilibrium $p\in\Delta_0$ are:
$$f_0(p)+f_1(p)+\dots+f_n(p),f_0(p)-f_1(p),f_0(p)-f_2(p), \dots, f_1(p)-f_n(p).$$

{\bf Local dynamics near equilibrium points}
%By carefully designing the local dynamics around each equilibrium and embedding the connections along prescribed arcs, we achieve a robust realization of the heteroclinic network. The construction ensures that each connection lies entirely within a 2D synchrony subspace.

Given a book embedding of a heteroclinic network $\mathcal{N}$ with thickness $k$ and $n$ nodes, the nodes of $\mathcal{N}$ are mapped to points in a straight-line $l\subset \mathbb{R}^3$.
%In the realization, we want to preserve the order between the points in the line given by the book embedding.
Without a loss of generality, we can assume that this line is equal to the $x$-axis, $l=\{(x,0,0):x\in\mathbb{R}\}$. 
Let $(\rho_i,0,0)$ be the image of a node $n_i$ by the book embedding and define the full-synchrony point $p_i=(\rho_i,\dots,\rho_i)\in\Delta_0\subset\mathbb{R}^{k+1}$.
%We assign to each node $n_i$ of the heteroclinic network a full-synchrony point $p_i\in\Delta_0\subset\mathbb{R}^{k+1}$, respecting the order between nodes given by the book embedding. 
Those points will correspond to equilibrium points in the realization of the heteroclinic network using a system associated to the coupled cell network $P_{k}$.
The plane $H_j$ of the book embedding will sketch the heteroclinic connections in the 2D synchrony subspace $\Delta_{j+1}$, when $j=1,\dots,k$.

For each point $p_i$, we choose constants $\alpha_0^i, \alpha_1^i, \dots \alpha_{k}^i$ such that $\alpha_0^i+ \alpha_1^i+ \dots +\alpha_{k}^i<0$, $\alpha_0^i-\alpha_{j}^i>0$ if there exists an outgoing edge starting in $p_i$ on the plane $H_j$ and $\alpha_0^i-\alpha_{j}^i<0$ otherwise, for $j=1,\dots,k$. 
The previous conditions are feasible because each plane $H_j$ has exclusively incoming or outgoing edges from a equilibrium point and the heteroclinic network has at least one outgoing connection from each node.
Take, for example, $\alpha_0^i=-1$, $\alpha_{j}^i=-2k$ if there is an outgoing edge from $p_i$ in the plane $H_j$ and $\alpha_{j}^i=1$ otherwise.

In a sufficient small neighborhood of each full-synchrony point $p_i\in\Delta_0\subset \mathbb{R}^{k+1}$, we define $f:\mathbb{R}^{k+1}\rightarrow \mathbb{R}$ such that 
$f(p_i)=0$ and the derivatives of $f$ at $p_i$ are equal to $\alpha_0^i,\alpha_1^i, \dots, \alpha_{k}^i$, i.e. $$f_j(p_i)=\alpha_j^i,$$
for $j=0,\dots,k$.
This means that $p_i$ is an equilibrium point of the coupled cell system $f^{P_{k}}$.
Moreover, there are trajectories escaping from $p_i$ into the synchrony subspace $\Delta_{j}$ only if there exists an outgoing edge leaving $n_i$ on the page $H_j$.
This function can be obtained as follows:
Take $\epsilon_1,\dots,\epsilon_n>0$ sufficient small such that $\overline{B_{2\epsilon_i}(p_i)}\cap \overline{B_{2\epsilon_j}(p_j)}=\emptyset$ for any $i\neq j$. 
And define
$$f(y_0,y_1,\dots,y_{k})=\sum_{i=1}^{n}\delta_{p_i}(y_0,y_1,\dots,y_{k}) \sum_{j=0}^{k} \alpha_j^{i}(y_j-\rho_i),$$
where $\delta_{p_i}$ is a bump function which is $1$ if $(y_0,y_1,\dots,y_{k})\in B_{\epsilon_i}(p_i)$ and it is $0$ if $(y_0,y_1,\dots,y_{k})\notin B_{2\epsilon_i}(p_i)$.
The coupled cell systems associated with $P_{k}$ has the form
$$\begin{cases}
\dot{x_0}=f(x_0,x_1,x_2, \dots,x_{k})\\
\dot{x_1}=f(x_1,x_0, x_2,\dots,x_{k})\\
\vdots\\
\dot{x_{k}}=f(x_{k},x_1,x_2,\dots,x_0)
\end{cases}$$
Note that the inputs of $f$ permute, so the vector field $f^{P_{k}}$ vanishes outside the balls $B_{2\epsilon_i}(p_i)$.
Moreover, inside each ball $B_{\epsilon_i}(p_i)$ the vector field $f^{P_{k}}$ is a linear map such that $p_i$ is a equilibrium point, $f^{P_{k}}(p_i)=0$, and the eigenvalues of $J_f^{P_{k}}(p_i)$ are 
$$\alpha_0^i+\alpha_1^i+\dots+\alpha_{k}^i, \alpha_0^i-\alpha_1^i,\alpha_0^i-\alpha_2^i, \dots, \alpha_0^i-\alpha_{k}^i.$$

Thus the vector field $f^{P_{k}}$ respects the local conditions, around the equilibrium points, to be a realization of the heteroclinic network, $\mathcal{N}$.
This means that if $n_i$ is equilibrium node with a outgoing connection on page $H_j$, then the equilibrium point $p_i$ has a unstable direction in $\Delta_j\setminus \Delta_0$.
And if $n_i$ is equilibrium node with a incoming connection on page $H_j$, then the equilibrium point $p_i$ is a sink in $\Delta_j$.

{\bf Realizations of heteroclinic connections}

In order to realize the heteroclinic connections we will change the coupled cell system outside a neighbourhood of the line $\Delta_0$.
Consider the embedding of a heteroclinic connection from $n_s$ to $n_t$ to an curve inside the plane $H_j$.
Let $\kappa>0$ such that $\kappa<\epsilon_s,\epsilon_t$ and the parallel translation by $\pm\kappa$ of $\Delta_0$ in $\Delta_{j}$ intersects transversely the unstable manifold of $p_s$. 
Note that the equilibrium points are stable in $\Delta_0$, since $\alpha_0^s+ \alpha_1^s+ \dots +\alpha_{k}^s<0$.
As $\alpha_0^s-\alpha_{j}^s>0$, the unstable manifold of $p_s$ in $\Delta_{j}$ transversely crosses a parallel translation by $\pm\kappa$ of $\Delta_0$, for $\kappa$ sufficiently small.
Moreover, the parallel translation by $\pm\kappa$ of $\Delta_0$ in $\Delta_{j}$ intersects the the stable manifold of $p_t$ since $\alpha_1^t+ \alpha_2^t+ \dots +\alpha_{k+1}^t,\alpha_1^t-\alpha_{j+1}^t<0$ and $\kappa<\epsilon_t$.
The parallel translation by $\pm\kappa$ of $\Delta_0$ in $\Delta_{j+1}$ is represented by the red line in Figure~\ref{fig:2dtraj}.
Since the heteroclinic netowkr is finite, we can choose $\kappa$ such that the previous hold for every heteroclinic connection.
Moreover, we can assume, without loss of generality, that the book embedding of each heteroclinic connection crosses the tubular neighborhood of $l$ with radius $\kappa$ exactly two time: one when escaping the starting node and another when approaching the targeting node.

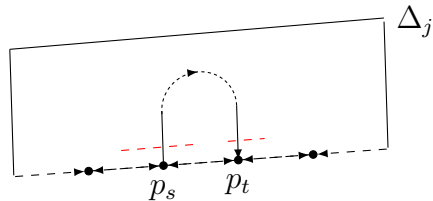
\begin{figure}[h]
\centering
\tdplotsetmaincoords{115}{70}
\begin{tikzpicture} [tdplot_main_coords,scale=0.5]

\tdplotsetrotatedcoords{60}{90}{90}
\node[tdplot_rotated_coords, inner sep=0, outer sep=0] (O) at (0,0,0) {};
\node (P1) [tdplot_rotated_coords,circle,draw,fill=black, inner sep=0pt, outer sep =0pt, minimum size=1mm] at (0,0,2) {};
\node (P2) [tdplot_rotated_coords,circle,draw,fill=black, inner sep=0pt, outer sep =0pt, minimum size=1mm, label=below:$p_s$] at (0,0,4) {};
\node (P3) [tdplot_rotated_coords,circle,draw,fill=black, inner sep=0pt, outer sep =0pt, minimum size=1mm, label=below:$p_t$] at (0,0,6) {};
\node (P4) [tdplot_rotated_coords,circle,draw,fill=black, inner sep=0pt, outer sep =0pt, minimum size=1mm] at (0,0,8) {};
\node[tdplot_rotated_coords, inner sep=0, outer sep=0] (A) at (0,0,10) {};
\draw[-latex, dashed] (O)  to (P1);
\draw[-latex, dashed] (P1) to (P2);
\draw[-latex, dashed] (P2) to (P1);
\draw[-latex, dashed] (P2) to (P3);
\draw[-latex, dashed] (P3) to (P2);
\draw[-latex, dashed] (P3) to (P4);
\draw[-latex, dashed] (P4) to (P3);
\draw[-latex, dashed] (A) to (P4);
\draw[ dashed, red] (1.5,2.46,0.5) to (2.5,4.11,0.5);

\draw[ dashed, red] (3,4.9,0.5) to (3.5,5.75,0.5);
\tdplotsetrotatedcoords{60}{90}{190}
\node (E2) [tdplot_rotated_coords,inner sep=0pt,outer sep =0pt,label=right:$\Delta_j$] at (3.5,0,10) {};
\draw[ tdplot_rotated_coords] (A)--(3.5,0,10);
\draw[ tdplot_rotated_coords] (3.5,0,0)--(3.6,0,10);
\draw[ tdplot_rotated_coords] (O)--(3.5,0,0);

\draw[ tdplot_rotated_coords] (0,0,4) to (1.5,0,4);
%\draw[dashed, tdplot_rotated_coords] (1.5,0,4) to (3,0,4);
\draw[-latex,  tdplot_rotated_coords] (1.5,0,6) to  (0,0,6);
\draw[latex-, tdplot_rotated_coords] ({1.5 + sin(90)},{0},{5 + cos(90)}) -- ({1.5 + sin(90 + 5)},{0},{5 + cos(90 + 5)});

\foreach \t in {0,10,...,180}
\draw[ tdplot_rotated_coords] ({1.5 + sin(\t)},{0},{5 + cos(\t)}) -- ({1.5 + sin(\t + 5)},{0},{5 + cos(\t + 5)});
%\draw[-latex,  tdplot_rotated_coords,dashed] (3,0,6) to  (1.5,0,6);

\end{tikzpicture}
\caption{Illustration of a heteroclinic trajectory connecting $n_s$ to $n_t$ 
embedded in the 2D synchrony subspace $\Delta_j$.}
\label{fig:2dtraj}
\end{figure}

Consider a heteroclinic connection from $n_s$ to $n_t$ that is mapped by the book embedding into an curve inside the plane $H_j$.
To simplify notation and provide an expression for the coupled cell system, we rotate the plane $H_j$ into the $xy$-plane. 
Let $\gamma:[0,1]\rightarrow \mathbb{R}^2$ be the book embedding of the heteroclinic connection rotated to the $xy$-plane.
Then the path $\gamma=(\gamma_1,\gamma_2)$ goes from $(\rho_s,0)$ to $(\rho_t,0)$ and it crosses the line $y=\pm \kappa$ twice. 
So there are $\tau_1$ and $\tau_2$ such that $\gamma_2(\tau_1)=\gamma_2(\tau_2)=\pm \kappa$ and $|\gamma_2(t)|> \kappa$ if and only if $\tau_1<t<\tau_2$.
Next, the heteroclinic connection will be realized in the coupled cell system $f^{P_{k}}$ by a trajectory following the arc $\psi:[\tau_1,\tau_2]\rightarrow \Delta_{j}\subset\mathbb{R}^{k+1}$ given by $(\psi)_i(t)=\gamma_1(t)$, $i\neq j$, and $(\psi)_{j}(t)=\gamma_1(t)+\gamma_2(t)$. 

We modify the function $f$ in such way that the coupled cell system $f^{P_{k}}$ is tangent to the arc $\psi:[\tau_1,\tau_2]\rightarrow \mathbb{R}^{k+1}$ outside the $\kappa$ tubular neighborhood of the full synchrony line $\Delta_0$, $B_\kappa (\Delta_0)$ . 
Let $A=B_{\epsilon}(\psi([\tau_1,\tau_2]))\setminus \overline{B_\kappa (\Delta_0)}$ and $B= B_{\epsilon}(\{(\gamma_1(t)+\gamma_2(t), \gamma_1(t),\dots,\gamma_1(t)): t\in[\tau_1,\tau_2]\})\setminus \overline{B_\kappa (\Delta_0)}$ be tubular neighborhoods. 
Choosing $\epsilon$ small enough we have that $A\cap B=\emptyset$. 
Moreover, for $\kappa$ small enough, we have that $W^u (p_s)\cap A\neq \emptyset$.
The sets $A$ and $B$ are tubular neighborhoods corresponding to the inputs in equation (\ref{eq:ccs2d}) of the function $f$.
So, we are going to change the function $f$ in these tubular neighborhoods using ``horizontal flows'' following the arcs $\psi$ and $\{(\gamma_1(t)+\gamma_2(t), \gamma_1(t),\dots,\gamma_1(t)): t\in[\tau_1,\tau_2]\}$. Let  $\delta_A$($\delta_B$) be a bump function which is equal to zero outside $A$($B$) and it is bigger than zero inside $A$($B$). 
We add the following term to $f$
$$\delta_A(y_0,y_1,\dots,y_k)v(y)+\delta_B(y_0,y_1,\dots,y_k)w(y),$$
where $v(y)$ is the derivative of $\gamma_1$ at the time $t$ where the projection of $y$ into $\psi([\tau_1,\tau_2])$ (the center of $A$) is equal to $\psi(t)$, when $y\in A$, and 
$w(y)$ is the derivative of $\gamma_1+\gamma_2$ at the time $t$ where the projection of $y$ into the center of $B$ is equal to $(\gamma_1(t)+\gamma_2(t), \gamma_1(t),\dots,\gamma_1(t))$, when $y\in B$. 

The coupled cell network $\dot{x}=f^{P_k}(x)$ realizes the heteroclinic connection from $n_s$ to $n_t$ for the trajectory starting in $W^u(p_s)\cap \Delta_j\cap B_\kappa(\Delta_0)\neq \emptyset$.
Such trajectory crosses the tubular neighborhood of $\Delta_0$, and then follows the direction of $\psi$ until it reaches again the tubular neighborhood of $\Delta_0$ near the point $p_t$.
Since the equilibrium point $p_t$ is a sink in $\Delta_j$, we know that this trajectory belongs to the stable manifold of $p_t$.
Figure~\ref{fig:2dtraj} sketches the heteroclinic connection constructed above. 
Moreover, the realization of the heteroclinic connection persists for small perturbations of the function $f$.

Now, we should repeat the process for the other heteroclinic connections.
For the other connections in the same page, there is no interference in realizing them as their embedding do not cross and we can select disjoint tubular neighborhoods.  
However, for connections in other pages, there can be interference if when we overlap two pages there is a crossing between connections on different pages.
Next, we take care of this interference by adjusting the book embedding of the connections when needed.

{\bf Pages and connections overlap}

%Similarly, we realize the other heteroclinic connections inductively. 
Take another heteroclinic connection from $n_{\tilde{s}}$ to $n_{\tilde{t}}$ which is book-embedded into the plane $H_{\tilde{j}}$ and define $\tilde{\gamma}$ to be the the book embedding of the heteroclinic connection rotated to the $xy$-plane.
%Let $\tilde{\tau}_1$ and $\tilde{\tau}_2$ be the instances where $\tilde{\gamma}$ intersect the line $y=\pm\kappa$ and stays outside the tubular neighborhood of $\Delta_0$, $|\tilde{\gamma}_2(t)|>\kappa$.
Depending if $\tilde{\gamma}$ intersects any of the previous arcs $\gamma$ or not, we adjust or not the book-embedding by changing the arc $\tilde{\gamma}$.
The intersection between these two arc do not depend if they are in the same page or not, as every arc is the rotation to the $xy$-plane.
In fact, the intersection can only occur if the tho connection belong to different pages, as the book embedding does not allow crossings.  
When there is no intersection, the tubular neighborhoods $\tilde{A}$ and $\tilde{B}$ are defined as before by taking a radius small enough to avoid intersections with previous defined tubular neighborhoods.
In this case, we add analogous terms to the function $f$ to realize the ``horizontal flow'' without changing the realization of the previous heteroclinic connections.  
So, this non-crossing heteroclinic connection from $n_{\tilde{s}}$ to $n_{\tilde{t}}$ is robustly realized in the coupled cell system $\dot{x}=f^{P_{k}}(x)$.
Next, we look to the case of a arc that intersect some of the previous arcs, in the $xy$-plane.

Suppose that the arc $\tilde{\gamma}$ intersects one of the previous arcs.
We can assume that each intersecting point belongs to exactly two arcs.
Moreover, the intersection occurs for $|y|>\kappa$, by also adjusting the book embedding.
Let $\gamma$ be the arc that intersects $\tilde{\gamma}$.
Note that the two arc intersecting must correspond to heteroclinic connections book-embedded in two different planes.
So the heteroclinic connection associated to $\gamma$ is book embedded in a page $H_j$, where $j\neq \tilde{j}$.
Let $\gamma^j$, $A$ and $B$ be as defined before to realize the heteroclinic connection in $\Delta_j$.
And define $\tilde{\gamma}^{\tilde{j}}$ the arc in $\Delta_{\tilde{j}}$, and the tubular neighborhoods $\tilde{A}$ and $\tilde{B}$ of $\tilde{\gamma}^{\tilde{j}}$ and $(\tilde{\gamma}_1+\tilde{\gamma}_2, \tilde{\gamma}_1,\dots,\tilde{\gamma}_1)$, respectively, for the arc $\tilde{\gamma}$.
Since $j\neq \tilde{j}$, we have that $A$ and $\tilde{A}$ do not intersect.
Looking to the equation (\ref{eq:ccs2d}), we see that the term $\delta_{\tilde{A}}(y)\tilde{v}(y)$ can be added to the function $f$ and the previous realized heteroclinic trajectories still exist. 
However, $B$ and $\tilde{B}$ do intersect. 
The derivatives of $\gamma_1+\gamma_2$ and $\tilde{\gamma}_1+\tilde{\gamma}_2$ should be both positive or both negative, i.e. $w(y) \tilde{w}(y)>0$, for $y \in B\cap \tilde{B}$.
If $w(y) \tilde{w}(y)>0$, then we can add the term $\delta_{\tilde{B}}(y)\tilde{w}(y)$ to the function $f$ without  perturbing the previous heteroclinic trajectories and realizing the heteroclinic connection associated with $\tilde{\gamma}$.
In order to finish the proof of Theorem~\ref{teo:bookembedding}, we see how we can adjust the arc $\tilde{\gamma}$ to ensure that $w(y) \tilde{w}(y)>0$.

Let $\tau_3$ and $\tau_4$ be times when the path $(\gamma_1(t)+\gamma_2(t), \gamma_1(t),\dots,\gamma_1(t))$ crosses the tubular neighborhood $\tilde{B}$ and pass through the intersection point. 
Similarly, let $\tilde{\tau}_3$ and $\tilde{\tau}_4$ be the times when the arc $(\tilde{\gamma}_1(t)+\tilde{\gamma}_2(t), \tilde{\gamma}_1(t),\dots,\tilde{\gamma}_1(t))$ to crosses $B$.
Without loss of generality, we can assume that $\dot{\gamma}_1$ and $\dot{\gamma}_2$ ($\dot{\tilde{\gamma}}_1$ and $\dot{\tilde{\gamma}}_2$) do not vanish between $\tau_3$ and $\tau_4$ ($\tilde{\tau}_3$ and $\tilde{\tau}_4$, respectively).
Figure~\ref{fig:adj2dtraj} displays the possibles cases: (a) $w(y) \tilde{w}(y)>0$;  (b) $w(y) \tilde{w}(y)<0$ and $\dot{\tilde{\gamma}}_1 \dot{\tilde{\gamma}}_2<0$; and (c) $w(y) \tilde{w}(y)<0$ and $\dot{\tilde{\gamma}}_1 \dot{\tilde{\gamma}}_2>0$.
For the case
\\
(a) we do not need to adjust the arc $\tilde{\gamma}$ as $w(y) \tilde{w}(y)>0$;\\
(b) we increase or decrease the speed of $\tilde{\gamma}_1$ to make $(\dot{\tilde{\gamma}}_1(t)+\dot{\tilde{\gamma}}_2(t))$ and $(\dot{\gamma}_1(t)+\dot{\gamma}_2(t))$ of the same sign, since $\dot{\tilde{\gamma}}_1 \dot{\tilde{\gamma}}_2<0$; \\
(c) we make smooth adjustments to the arc $\tilde{\gamma}$ in the interval $[\tilde{\tau}_3,\tilde{\tau}_4]$ such that the adjusted arc is arbitrarily close to the union of straight lines: $\tilde{\gamma}(\tilde{\tau}_3)$--$(\tilde{\gamma}_1(\tilde{\tau}_4),\tilde{\gamma}_2(\tilde{\tau}_3))$--
$(\tilde{\gamma}_1(\tilde{\tau}_3),\tilde{\gamma}_2(\tilde{\tau}_4))$--$\tilde{\gamma}(\tilde{\tau}_4)$. 
In the adjusted arc, the signs of $\dot{\tilde{\gamma}}_2$ is inverted, around a neighborhood of the intersection point.
So, we can speed up or down $\tilde{\gamma}_1$ and match the signals of $(\dot{\tilde{\gamma}}_1(t)+\dot{\tilde{\gamma}}_2(t))$ and $(\dot{\gamma}_1(t)+\dot{\gamma}_2(t))$, as in case (b).
Figure~\ref{fig:adj2dtraj} sketches the adjustments to $\tilde{\gamma}$ made above.  

%Once we have this match between the sign we can add the component $\delta_{\tilde{B}}(x_0,x_1,\dots,x_k)(\dot{\tilde{\gamma}}_1(t)+\dot{\tilde{\gamma}}_2(t))$ to $f$ without destroying the previous heteroclinic trajectories and realizing the heteroclinic connection from $n_{\tilde{s}}$ to $n_{\tilde{t}}$ in $\Delta_{\tilde{j}}$. 

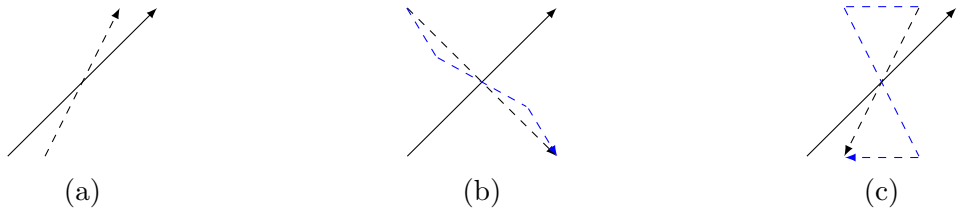
\begin{figure}[h]
\centering
\begin{subfigure}{.32\textwidth}
\centering
\begin{tikzpicture} [scale=2]
\node[inner sep=0, outer sep=0] (P1) at (0,0) {};
\node[inner sep=0, outer sep=0] (P2) at (1,1) {};
\node[inner sep=0, outer sep=0] (P3) at (0.25,0) {};
\node[inner sep=0, outer sep=0] (P4) at (0.75,1) {};

\draw[-latex, black] (P1)  to (P2);
\draw[-latex, dashed] (P3)  to (P4);
\end{tikzpicture}
\caption{}
\label{fig:adj2dtraja}
\end{subfigure}
\begin{subfigure}{.32\textwidth}
\centering
\begin{tikzpicture} [scale=2]
\node[inner sep=0, outer sep=0] (P1) at (0,0) {};
\node[inner sep=0, outer sep=0] (P2) at (1,1) {};
\node[inner sep=0, outer sep=0] (P3) at (0,1) {};
\node[inner sep=0, outer sep=0] (P4) at (1,0) {};
\node[inner sep=0, outer sep=0] (P5) at (1/5,2/3) {};
\node[inner sep=0, outer sep=0] (P6) at (4/5,1/3) {};

\draw[-latex] (P1)  to (P2);
\draw[-latex, dashed] (P3)  to (P4);

\draw[ dashed, blue] (P3) to (P5);
\draw[ dashed, blue] (P5) to (P6);
\draw[-latex, dashed, blue] (P6) to (P4);
\end{tikzpicture}
\caption{}
\label{fig:adj2dtrajb}
\end{subfigure}
%\begin{subfigure}{.24\textwidth}
%\begin{tikzpicture} [scale=2]
%\node[inner sep=0, outer sep=0] (P1) at (0,0) {};
%\node[inner sep=0, outer sep=0] (P2) at (1,1) {};
%\node[inner sep=0, outer sep=0] (P3) at (1,0) {};
%\node[inner sep=0, outer sep=0] (P4) at (0,1) {};
%\node[inner sep=0, outer sep=0] (P5) at (4/5,1/3) {};
%\node[inner sep=0, outer sep=0] (P6) at (1/5,2/3) {};
%
%\draw[line width=1mm, gray!60] (P1) to (P2);
%\draw[line width=1mm, gray!60] (P3)  to (P4);
%\draw[line width=1mm, gray!30] (P3) to (P5);
%\draw[line width=1mm, gray!30] (P5) to (P6);
%\draw[line width=1mm, gray!30] (P6) to (P4);
%
%
%\draw[-latex] (P1)  to (P2);
%\draw[-latex, dashed] (P3)  to (P4);
%\draw[ dashed, blue] (P3) to (P5);
%\draw[ dashed, blue] (P5) to (P6);
%\draw[-latex, dashed, blue] (P6) to (P4);
%\end{tikzpicture}
%\caption{}
%\label{fig:adj2dtrajc}
%\end{subfigure}
\begin{subfigure}{.32\textwidth}
\centering
\begin{tikzpicture} [scale=2]
\node[inner sep=0, outer sep=0] (P1) at (0,0) {};
\node[inner sep=0, outer sep=0] (P2) at (1,1) {};
\node[inner sep=0, outer sep=0] (P3) at (0.75,1) {};
\node[inner sep=0, outer sep=0] (P4) at (0.25,0) {};
\node[inner sep=0, outer sep=0] (P5) at (0.25,1) {};
\node[inner sep=0, outer sep=0] (P6) at (0.75,0) {};

\draw[-latex] (P1)  to (P2);
\draw[-latex, dashed] (P3)  to (P4);
\draw[ dashed, blue] (P3) to (P5);
\draw[ dashed, blue] (P5) to (P6);
\draw[-latex, dashed, blue] (P6) to (P4);
\end{tikzpicture}
\caption{}
\label{fig:adj2dtrajd}
\end{subfigure}
\caption{Sketch of the adjustments made to connections that intersect when overlapping book embedding pages. In the different cases we make the following adjustment to $\tilde{\gamma}$ from the black dashed lines to the blue dashed lines: (a) No adjustment needed; (b) Adjust the speed of $\tilde{\gamma}_1$; (c) Modify $\tilde{\gamma}_1$ to invert the sign of $\dot{\tilde{\gamma}}_1$ at the intersection point.}
\label{fig:adj2dtraj}
\end{figure}

Repeating the previous procedure for every heteroclinic connection, we obtain a robust realization of the heteroclinic network in a coupled cell system with $k+1$ cells, $\dot{x}=f^{P_{k}}(x)$. This concludes the proof of Theorem~\ref{teo:bookembedding}. 
Note that the trajectories belong to the 2D synchrony subspace $\Delta_j$, so, along the trajectory, every cell except the cell $j$ remain in synchrony and the cell $j$ re synchronizes with the others at the end of the heteroclinic connection.

\subsection{Examples and Remarks}

To illustrate Theorem~\ref{teo:bookembedding}, we present examples and discuss implications of the book-thickness constraint. 
We also compare our results with existing constructions in the literature.

\begin{exe}
The heteroclinic network in Figure~\ref{fig:hetnet8} has book-thickness equal to $2$, as seen in Example~\ref{ex:hetnet8}.
Thus, it follows from Theorem~\ref{teo:bookembedding}, that this heteroclinic network can be robustly realized in a coupled cell system with $3$ cells, i.e. in $\mathbb{R}^3$.
\end{exe}

Double-next-neighbor heteroclinic networks were studied by Castro and Lohse \cite{CL23}.
They gave an explicit system of ODE-equations in $\mathbb{R}^6$ realizing this network.
In the next example, we look to these heteroclinic networks and check that they can be robustly realized in coupled cell system with dimension $4$, $5$ or $6$ depending on the number of equilibrium points.

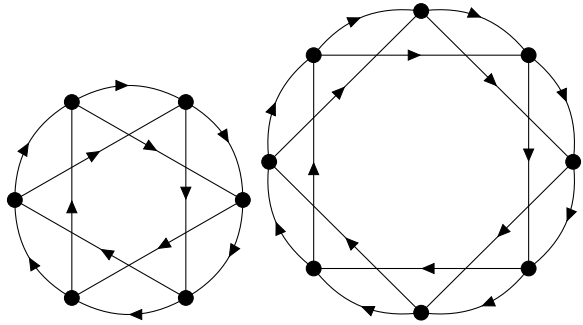
\begin{figure}[h]
\centering

\begin{tikzpicture}[every node/.style={draw,shape=circle,fill=black,scale=0.5},    midarrow/.style={
        postaction={decorate},
        decoration={markings, mark=at position 0.5 with {\arrow[scale=0.8,fill=black]{triangle 45}}}}]

% Number of nodes
\def\n{6}
\def\radius{1.5}
\pgfmathsetmacro{\nmu}{\n-1}

\foreach \i in {0,...,\nmu} {
		\pgfmathsetmacro{\j}{mod(\i-1,\n)}
    \node (N\i) at ({\radius*cos(360/\n*\i)}, {\radius*sin(360/\n*\i)}) {};
}

% Connect to nearest neighbors
\foreach \i in {0,...,\nmu} {
    \pgfmathsetmacro{\j}{mod(\i-1,\n)}
    \pgfmathsetmacro{\k}{mod(\i-2,\n)}
    \draw ({\radius*cos(360/\n*\i)}, {\radius*sin(360/\n*\i)}) [midarrow] to [bend left]  ({\radius*cos(360/\n*\j)}, {\radius*sin(360/\n*\j)});
		\draw ({\radius*cos(360/\n*\i)}, {\radius*sin(360/\n*\i)}) [midarrow] to   ({\radius*cos(360/\n*\k)}, {\radius*sin(360/\n*\k)});
}
\end{tikzpicture}
\begin{tikzpicture}[every node/.style={draw,shape=circle,fill=black,scale=0.5},    midarrow/.style={
        postaction={decorate},
        decoration={markings, mark=at position 0.5 with {\arrow[scale=0.8,fill=black]{triangle 45}}}}]

% Number of nodes
\def\n{8}
\def\radius{2}
\pgfmathsetmacro{\nmu}{\n-1}

\foreach \i in {0,...,\nmu} {
		\node (N\i) at ({\radius*cos(360/\n*\i)}, {\radius*sin(360/\n*\i)}) {};
}

% Connect to nearest neighbors
\foreach \i in {0,...,\nmu} {
    \pgfmathsetmacro{\j}{mod(\i-1,\n)}
    \pgfmathsetmacro{\k}{mod(\i-2,\n)}
    \draw ({\radius*cos(360/\n*\i)}, {\radius*sin(360/\n*\i)}) [midarrow] to [bend left]  ({\radius*cos(360/\n*\j)}, {\radius*sin(360/\n*\j)});
		\draw ({\radius*cos(360/\n*\i)}, {\radius*sin(360/\n*\i)}) [midarrow] to   ({\radius*cos(360/\n*\k)}, {\radius*sin(360/\n*\k)});
}
\end{tikzpicture}
\caption{Some double-next-neighbor networks.}
\label{fig:dnnet}
\end{figure}

\begin{exe}
A network with $n$ nodes is a double-next-neighbor, if there exists a ordering of the cells such that each node receives two connections from the two preceding nodes module $n$. 
Some examples are displayed in Figure~\ref{fig:dnnet}.

We consider the following book embedding of the double-next-neighbor network. 
The equilibrium nodes are placed in a straight line.
The connection from $1$ to $2$ goes to page $H_1$, the connection from $1$ to $3$ goes to page $H_2$, the connection from $2$ to $3$ goes to page $H_2$, the connection from $2$ to $4$ goes to page $H_3$,  the connection from $3$ to $4$ goes to page $H_3$, the connection from $3$ to $5$ goes to page $H_1$, and we repeat this process until the outgoing connections from node $n-2$.
The book embedding of the previous connections can be done free of intersections by placing the arcs successively above and below straight line.
We place the connection from $n-1$ to $n$ in the same page that the connection from $n-2$ to $n$, lets say $H_k$.
Note that $k$ depends on the number of equilibrium node, $n$, and the last $3$ connections are placed depending on $k$.\\
If $k=1$ ($n$ module 3 is $2$), then we place the connections from $n-1$ to $1$ and from $n$ to $1$ in a new page $H_4$ and the connection from $n$ to $2$ in page $5$.\\
If $k=2$ ($n$ module 3 is $0$), the connections from $n-1$ to $1$ and from $n$ to $1$ in page $3$ and the connection from $n$ to $2$ in page $1$.\\
If $k=3$ ($n$ module 3 is $1$), the connections from $n-1$ to $1$ and from $n$ to $1$ in page $4$ and the connection from $n$ to $2$ in page $1$.\\
In Figure~\ref{fig:ddnemb}, we display this book embedding using different colors for the different pages where the connections are embedded.
Thus,it follows from Theorem~\ref{teo:bookembedding} that, double-next-neighbor heteroclinic networks can be realized using coupled cell systems with four, five or six cells, depending on the number of equilibrium nodes.
\end{exe}

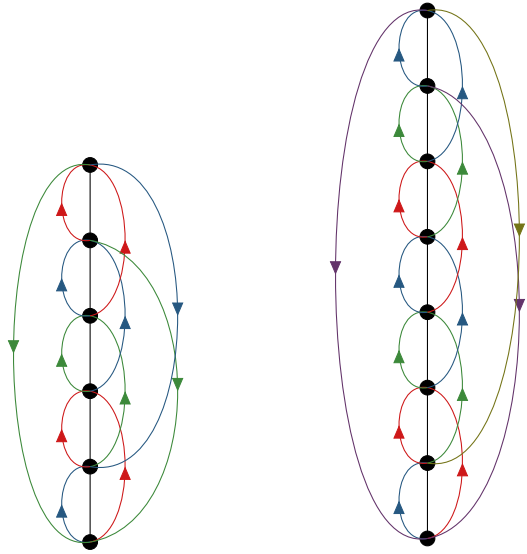
\begin{figure}[h]
\centering
\begin{tikzpicture}[every node/.style={draw,shape=circle,fill=black,scale=0.5},    midarrow/.style={
        postaction={decorate},
        decoration={markings, mark=at position 0.5 with {\arrow[scale=0.8,fill=black]{triangle 45}}}}]

% Number of nodes
\def\n{6}
\pgfmathsetmacro{\nmu}{\n-1}

% Draw nodes in a straight line
\foreach \i in {0,...,\nmu} {
    \node (N\i) at (0, \i*1) { };
}
    \draw (0,0) to (0, \nmu) { };

\draw[draw={rgb:red,55;green,126;blue,184}] (0, 0) [midarrow] to [bend left=100, looseness=1.3]  (0, 1);
\draw[draw={rgb:red,228;green,26;blue,28}] (0, 0) [midarrow] to [bend right=80, looseness=0.8]  (0, 2);

\draw[draw={rgb:red,228;green,26;blue,28}] (0, 1) [midarrow] to [bend left=100, looseness=1.3]  (0, 2);
\draw[draw={rgb:red,77;green,175;blue,74}] (0, 1) [midarrow] to [bend right=80, looseness=0.8]  (0, 3);

\draw[draw={rgb:red,77;green,175;blue,74}] (0, 2) [midarrow] to [bend left=100, looseness=1.3]  (0, 3);
\draw[draw={rgb:red,55;green,126;blue,184}] (0, 2) [midarrow] to [bend right=80, looseness=0.8]  (0, 4);

\draw[draw={rgb:red,55;green,126;blue,184}] (0, 3) [midarrow] to [bend left=100, looseness=1.3]  (0, 4);
\draw[draw={rgb:red,228;green,26;blue,28}] (0,3) [midarrow] to [bend right=80, looseness=0.8]  (0, 5);

\draw[draw={rgb:red,228;green,26;blue,28}] (0, 4) [midarrow] to [bend left=100, looseness=1.3]  (0, 5);
\draw[draw={rgb:red,77;green,175;blue,74}] (0, 4) [midarrow] to [bend left=80, looseness=1]  (0, 0);

\draw[draw={rgb:red,77;green,175;blue,74}] (0, 5) [midarrow] to [bend right=100, looseness=0.7]  (0, 0);
\draw[draw={rgb:red,55;green,126;blue,184}] (0,5) [midarrow] to [bend left=100, looseness=1]  (0, 1);
%
%\draw[draw={rgb:red,55;green,126;blue,184}] (0,6) [midarrow] to [bend left=100, looseness=1.3]  (0, 7);
%\draw[draw={rgb:red,152;green,78;blue,163}]  (0,6) [midarrow] to [bend right=80, looseness=0.7]  (0, 0);
%
%\draw[draw={rgb:red,152;green,78;blue,163}] (0,7) [midarrow] to [bend left=100, looseness=0.6]  (0, 0);
%\draw[draw={rgb:red,255;green,255;blue,51}]  (0,7) [midarrow] to [bend right=100, looseness=0.7]  (0, 1);
\end{tikzpicture}
\hspace{10mm}
\begin{tikzpicture}[every node/.style={draw,shape=circle,fill=black,scale=0.5},    midarrow/.style={
        postaction={decorate},
        decoration={markings, mark=at position 0.5 with {\arrow[scale=0.8,fill=black]{triangle 45}}}}]

% Number of nodes
\def\n{8}
\pgfmathsetmacro{\nmu}{\n-1}

% Draw nodes in a straight line
\foreach \i in {0,...,\nmu} {
    \node (N\i) at (0, \i*1) { };
}
    \draw (0,0) to (0, \nmu) { };

\draw[draw={rgb:red,55;green,126;blue,184}] (0, 0) [midarrow] to [bend left=100, looseness=1.3]  (0, 1);
\draw[draw={rgb:red,228;green,26;blue,28}] (0, 0) [midarrow] to [bend right=80, looseness=0.8]  (0, 2);

\draw[draw={rgb:red,228;green,26;blue,28}] (0, 1) [midarrow] to [bend left=100, looseness=1.3]  (0, 2);
\draw[draw={rgb:red,77;green,175;blue,74}] (0, 1) [midarrow] to [bend right=80, looseness=0.8]  (0, 3);

\draw[draw={rgb:red,77;green,175;blue,74}] (0, 2) [midarrow] to [bend left=100, looseness=1.3]  (0, 3);
\draw[draw={rgb:red,55;green,126;blue,184}] (0, 2) [midarrow] to [bend right=80, looseness=0.8]  (0, 4);

\draw[draw={rgb:red,55;green,126;blue,184}] (0, 3) [midarrow] to [bend left=100, looseness=1.3]  (0, 4);
\draw[draw={rgb:red,228;green,26;blue,28}] (0,3) [midarrow] to [bend right=80, looseness=0.8]  (0, 5);

\draw[draw={rgb:red,228;green,26;blue,28}] (0, 4) [midarrow] to [bend left=100, looseness=1.3]  (0, 5);
\draw[draw={rgb:red,77;green,175;blue,74}] (0, 4) [midarrow] to [bend right=80, looseness=0.8]  (0, 6);

\draw[draw={rgb:red,77;green,175;blue,74}] (0, 5) [midarrow] to [bend left=100, looseness=1.3]  (0, 6);
\draw[draw={rgb:red,55;green,126;blue,184}] (0,5) [midarrow] to [bend right=80, looseness=0.8]  (0, 7);

\draw[draw={rgb:red,55;green,126;blue,184}] (0,6) [midarrow] to [bend left=100, looseness=1.3]  (0, 7);
\draw[draw={rgb:red,152;green,78;blue,163}]  (0,6) [midarrow] to [bend left=80, looseness=0.7]  (0, 0);

\draw[draw={rgb:red,152;green,78;blue,163}] (0,7) [midarrow] to [bend right=100, looseness=0.6]  (0, 0);
\draw[draw={rgb:red,255;green,255;blue,51}]  (0,7) [midarrow] to [bend left=100, looseness=0.7]  (0, 1);
\end{tikzpicture}
\caption{Book embedding of some double-next-neighbor networks. Different colors of the connections represents different pages where they are embedded.}
\label{fig:ddnemb}
\end{figure}

In the previous example, the two incoming connections of almost all equilibrium nodes belong to the same page. 
However, we can also book embed double-next-neighbor networks by placing the two outgoing connections from each equilibrium node in the same page.
In this way, the unstable manifold at each equilibrium point is one dimensional and it is fully contained in the heteroclinic network.
So, there are coupled cell systems that realize any double-next-neighbor heteroclinic network and this realization is complete.

\begin{exe}
Figure~\ref{fig:ddnemb2} display book-embedding of two double-next-neighbor heteroclinic networks such that the outgoing connections from each equilibrium node share the same page.
We can always find such book-embedding using $5$ pages.
The unstable manifold of each equilibrium point belong to a unique 2D synchrony subspace and it is one dimensional.
As each equilibrium point has two outgoing connections, its unstable manifold is fully contained in the heteroclinic network. 
There exists coupled cell systems with $6$ cells that realize the double-next-neighbor heteroclinic networks in a complete way.
\end{exe}

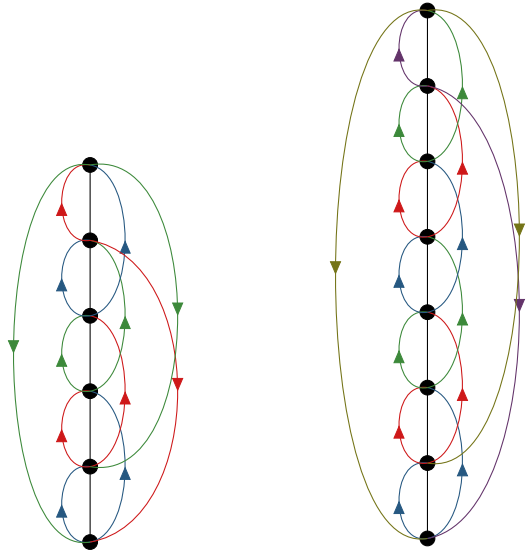
\begin{figure}[h]
\centering
\begin{tikzpicture}[every node/.style={draw,shape=circle,fill=black,scale=0.5},    midarrow/.style={
        postaction={decorate},
        decoration={markings, mark=at position 0.5 with {\arrow[scale=0.8,fill=black]{triangle 45}}}}]

% Number of nodes
\def\n{6}
\pgfmathsetmacro{\nmu}{\n-1}

% Draw nodes in a straight line
\foreach \i in {0,...,\nmu} {
    \node (N\i) at (0, \i*1) { };
}
    \draw (0,0) to (0, \nmu) { };

\draw[draw={rgb:red,55;green,126;blue,184}] (0, 0) [midarrow] to [bend left=100, looseness=1.3]  (0, 1);
\draw[draw={rgb:red,55;green,126;blue,184}] (0, 0) [midarrow] to [bend right=80, looseness=0.8]  (0, 2);

\draw[draw={rgb:red,228;green,26;blue,28}] (0, 1) [midarrow] to [bend left=100, looseness=1.3]  (0, 2);
\draw[draw={rgb:red,228;green,26;blue,28}] (0, 1) [midarrow] to [bend right=80, looseness=0.8]  (0, 3);

\draw[draw={rgb:red,77;green,175;blue,74}] (0, 2) [midarrow] to [bend left=100, looseness=1.3]  (0, 3);
\draw[draw={rgb:red,77;green,175;blue,74}] (0, 2) [midarrow] to [bend right=80, looseness=0.8]  (0, 4);

\draw[draw={rgb:red,55;green,126;blue,184}] (0, 3) [midarrow] to [bend left=100, looseness=1.3]  (0, 4);
\draw[draw={rgb:red,55;green,126;blue,184}] (0,3) [midarrow] to [bend right=80, looseness=0.8]  (0, 5);

\draw[draw={rgb:red,228;green,26;blue,28}] (0, 4) [midarrow] to [bend left=100, looseness=1.3]  (0, 5);
\draw[draw={rgb:red,228;green,26;blue,28}] (0, 4) [midarrow] to [bend left=80, looseness=1]  (0, 0);

\draw[draw={rgb:red,77;green,175;blue,74}] (0, 5) [midarrow] to [bend right=100, looseness=0.7]  (0, 0);
\draw[draw={rgb:red,77;green,175;blue,74}] (0,5) [midarrow] to [bend left=100, looseness=1]  (0, 1);
%
%\draw[draw={rgb:red,55;green,126;blue,184}] (0,6) [midarrow] to [bend left=100, looseness=1.3]  (0, 7);
%\draw[draw={rgb:red,152;green,78;blue,163}]  (0,6) [midarrow] to [bend right=80, looseness=0.7]  (0, 0);
%
%\draw[draw={rgb:red,152;green,78;blue,163}] (0,7) [midarrow] to [bend left=100, looseness=0.6]  (0, 0);
%\draw[draw={rgb:red,255;green,255;blue,51}]  (0,7) [midarrow] to [bend right=100, looseness=0.7]  (0, 1);
\end{tikzpicture}
\hspace{10mm}
\begin{tikzpicture}[every node/.style={draw,shape=circle,fill=black,scale=0.5},    midarrow/.style={
        postaction={decorate},
        decoration={markings, mark=at position 0.5 with {\arrow[scale=0.8,fill=black]{triangle 45}}}}]

% Number of nodes
\def\n{8}
\pgfmathsetmacro{\nmu}{\n-1}

% Draw nodes in a straight line
\foreach \i in {0,...,\nmu} {
    \node (N\i) at (0, \i*1) { };
}
    \draw (0,0) to (0, \nmu) { };

\draw[draw={rgb:red,55;green,126;blue,184}] (0, 0) [midarrow] to [bend left=100, looseness=1.3]  (0, 1);
\draw[draw={rgb:red,55;green,126;blue,184}] (0, 0) [midarrow] to [bend right=80, looseness=0.8]  (0, 2);

\draw[draw={rgb:red,228;green,26;blue,28}] (0, 1) [midarrow] to [bend left=100, looseness=1.3]  (0, 2);
\draw[draw={rgb:red,228;green,26;blue,28}] (0, 1) [midarrow] to [bend right=80, looseness=0.8]  (0, 3);

\draw[draw={rgb:red,77;green,175;blue,74}] (0, 2) [midarrow] to [bend left=100, looseness=1.3]  (0, 3);
\draw[draw={rgb:red,77;green,175;blue,74}] (0, 2) [midarrow] to [bend right=80, looseness=0.8]  (0, 4);

\draw[draw={rgb:red,55;green,126;blue,184}] (0, 3) [midarrow] to [bend left=100, looseness=1.3]  (0, 4);
\draw[draw={rgb:red,55;green,126;blue,184}] (0,3) [midarrow] to [bend right=80, looseness=0.8]  (0, 5);

\draw[draw={rgb:red,228;green,26;blue,28}] (0, 4) [midarrow] to [bend left=100, looseness=1.3]  (0, 5);
\draw[draw={rgb:red,228;green,26;blue,28}] (0, 4) [midarrow] to [bend right=80, looseness=0.8]  (0, 6);

\draw[draw={rgb:red,77;green,175;blue,74}] (0, 5) [midarrow] to [bend left=100, looseness=1.3]  (0, 6);
\draw[draw={rgb:red,77;green,175;blue,74}] (0,5) [midarrow] to [bend right=80, looseness=0.8]  (0, 7);

\draw[draw={rgb:red,152;green,78;blue,163}] (0,6) [midarrow] to [bend left=100, looseness=1.3]  (0, 7);
\draw[draw={rgb:red,152;green,78;blue,163}]  (0,6) [midarrow] to [bend left=80, looseness=0.7]  (0, 0);

\draw[draw={rgb:red,255;green,255;blue,51}] (0,7) [midarrow] to [bend right=100, looseness=0.6]  (0, 0);
\draw[draw={rgb:red,255;green,255;blue,51}]  (0,7) [midarrow] to [bend left=100, looseness=0.7]  (0, 1);
\end{tikzpicture}
\caption{Book embedding of some double-next-neighbor networks. Different colors of the connections represents different pages where they are embedded.}
\label{fig:ddnemb2}
\end{figure}

Although planar graphs can be embedded in four pages in general, the additional constraints that we included may require more pages. 
The following remark highlights that the dynamical context imposes stricter embedding rules.

\begin{rem}
Calculating the book-thickness of a graph is a NP-problem. 
However, it is know that any planar graph can be embedded in at most $4$ pages. 
Our definition is more restrictive and a planar heteroclinic network does not need to be book-embedded in four pages.
For example a heteroclinic network formed by four 2 heteroclinic cycles with a common node needs at least 5 pages, because the common node has four outgoing connections that need to be in four different pages and the incoming connection need to be in a different page.
The two outgoing connections from each page are embedded in the same page. 
\end{rem}

%----------------
%
%Supondo que estamos a realizar um ciclo heteroclinico no Teorema~\ref{teo:bookembedding}, conseguimos aplicar este teu resultado?
%
%
%Theorem 3.4. Let $M_j$, $j=1,\dots,m$, be basic transition matrices of a collection of maps associated with a heteroclinic cycle.
%Suppose that for all $j=1,\dots,m$ all entries of the matrices are non-negative. Then:
%\\ (a) If the transition matrix $M(1) = M_m . . . M_1$ satisfies $|\lambda_m| > 1$ then
%$\sigma_j = +\infty$ for all $j=1,\dots,m$, and the cycle is asymptotically stable.
%\\ (b) Otherwise, $\sigma_j = -\infty$ for all $j=1,\dots,m$ and the cycle is not an attractor.
%
%
%Como calculamos o $M_i$?

\section{Almost complete realizations}

In this section, we prove that there are almost complete realizations of any heteroclinic network without homoclinic connections using a coupled cell system.
In this case, we realize some trajectories on 3D-synchrony subspaces and every outgoing connection from a node belong to the same synchrony subspace.

\begin{teo}\label{teo:almostcomplete}
Let $\mathcal{N}$ be a heteroclinic network without homoclinic connections where $n_1$ nodes have one or two outgoing connections and $n_2$ nodes have three or more outgoing connections. 
Then the heteroclinic network can be robustly realized in a coupled cell system with $n_1+2n_2+1$ cells. Moreover, this realization is almost complete.
\end{teo}

This result extends the previous realization by allowing connections in 3D synchrony subspaces when a node has multiple outgoing connections.

The proof of this result follows the same steps that the proof of Theorem~\ref{teo:bookembedding}.
First, we inductively construct a coupled cell network supporting the desired minimal synchrony subspaces and the convenient eigenvalues at full-synchronous equilibrium points. 
The equilibrium nodes of the heteroclinic network will correspond to equilibrium points inside the full-synchronous subspace. 
Given a node with one or two outgoing connections, its outgoing heteroclinic connections will be in the same 2D-synchrony subspace.
As done before, we embed the outgoing connections from that node in a page, $\mathbb{R}^2$, and use that to realize the outgoing heteroclinic connections from the respective equilibrium point.  
Taking a node with three or more outgoing connections, its outgoing heteroclinic connections will be in the same 3D-synchrony subspace.
In this case, the outgoing connections are embedded in $\mathbb{R}^3$ which is used to realize these heteroclinic connections.
The connections are embedded in a way that the obtained realization is almost complete.

{\bf Construction of the coupled cell network $Q_{n_1,n_2}$.}

Now, we inductively define the coupled cell network $Q_{n_1,n_2}$ with $n_1+2n_2+1$ cells.
The coupled cell network that we consider depends if there exists a node with less than two outgoing edges, $n_1>0$, or not, $n_1=0$.

If $n_1=0$, let $Q_{0,0}$ be the coupled cell network with one cell, called $0$, and no edges.
Assume that the coupled cell network $Q_{0,k}$ is known, we inductively define the network $Q_{0,k+1}$ by adding two new cells as follows.
The set of cells of $Q_{0,k}$ is  $\{0,1,\dots, 2k\}$ and there are $2k$ edges types divided as $E_1, E_2,\dots, E_{2k}$.
The cells of $Q_{0,k+1}$ are $\{0,1,\dots, 2k, 2k+1, 2(k+1)\}$ and it has $2(k+1)$ edge types.
For each $j=1,\dots,2k$, edges of type $i_j$ are $E_j\cup \{ (j,2k+1),(j,2(k+1))\}$.
The edges of type $i_{2k+1}$ are $\{(2k+1,c):c=0,\dots,2k\}\cup\{(2(k+1),2k+1),(0,2(k+1))\}$.
And the edges of type $i_{2(k+1)}$ are $\{(2k+2,c):c=0,\dots,2k\}\cup\{(2(k+1),2k+1),(2k+1,2(k+1))\}$.
Figure~\ref{fig:pn1} displays some examples.

\begin{figure}[h]
\centering
\begin{subfigure}{.2\textwidth}
\begin{tikzpicture}
\node (n1) [circle,draw]   {0};
\node (n2) [circle,draw] [below=of n1]  {1};
\node (n3) [circle,draw] [right=of n2]  {2};

%\draw[->, thick,loop left] (n1) to (n1);
%\draw[->, thick] (n1) to  (n2);
%\draw[->, thick] (n1) to  (n3);

\draw[->, thick] (n2) to (n1);
\draw[->, thick] (n3) to [bend right] (n2);
\draw[->, thick] (n1) to [bend left] (n3);

\draw[->>, thick] (n3) to [bend right] (n1);
\draw[->>, thick] (n3) to (n2);
\draw[->>, thick] (n2) to (n3);
\end{tikzpicture}
\end{subfigure}
\begin{subfigure}{.30\textwidth}
\begin{tikzpicture}
\node (n1) [circle,draw]   {0};
\node (n2) [circle,draw] [below=of n1]  {1};
\node (n3) [circle,draw] [right=of n2]  {2};
\node (n4) [circle,draw] [below=of n2]  {3};
\node (n5) [circle,draw] [right=of n4]  {4};

%\draw[->, thick,loop left] (n1) to (n1);
%\draw[->, thick] (n1) to  (n2);
%\draw[->, thick] (n1) to  (n3);
%\draw[->, thick] (n1) to [out=225,in=180] (n4);
%\draw[->, thick] (n1) to [out=45,in=45] (n5);

\draw[->, thick] (n2) to (n1);
\draw[->, thick] (n3) to [bend right] (n2);
\draw[->, thick] (n1) to [bend left] (n3);
\draw[->, thick] (n2) to (n4);
\draw[->, thick] (n2) to (n5);

\draw[->>, thick] (n3) to [bend right] (n1);
\draw[->>, thick] (n3) to (n2);
\draw[->>, thick] (n2) to (n3);
\draw[->>, thick] (n3) to [bend right] (n4);
\draw[->>, thick] (n3) to (n5);

\draw[->, dashed] (n4) to [bend left] (n1);
\draw[->, dashed] (n4) to [bend right] (n2);
\draw[->, dashed] (n4) to (n3);
\draw[->, dashed] (n5) to [bend right] (n4);
\draw[->, dashed] (n1) to (n5);

\draw[->>, dashed] (n5) to (n1);
\draw[->>, dashed] (n5) to [out=110,in=340] (n2);
\draw[->>, dashed] (n5) to [bend right] (n3);
\draw[->>, dashed] (n5) to (n4);
\draw[->>, dashed] (n4) to (n5);
\end{tikzpicture}
\end{subfigure}
\begin{subfigure}{.30\textwidth}
\begin{tikzpicture}
\node (n1) at ({2*sin(0)},{2*cos(0)}) [circle,draw]   {0};
\node (n2)  at ({2*sin(51)},{2*cos(51)}) [circle,draw]  {1};
\node (n3) [circle,draw] at ({2*sin(102)},{2*cos(102)})  {2};
\node (n4) [circle,draw] at ({2*sin(153)},{2*cos(153)})  {3};
\node (n5) [circle,draw] at ({2*sin(204)},{2*cos(204)})  {4};
\node (n6) [circle,draw] at ({2*sin(255)},{2*cos(255)})  {5};
\node (n7) [circle,draw] at ({2*sin(306)},{2*cos(306)})  {6};

%\draw[->, thick,loop above] (n1) to (n1);
%\draw[->, thick] (n1) to [bend left] (n2);
%\draw[->, thick] (n1) to  (n3);
%\draw[->, thick] (n1) to  (n4);
%\draw[->, thick] (n1) to  (n5);
%\draw[->, thick] (n1) to  (n6);
%\draw[->, thick] (n1) to  (n7);

\draw[->, thick] (n2) to (n1);
\draw[->, thick] (n3) to  (n2);
\draw[->, thick] (n1) to [out=325,in=115] (n3);
\draw[->, thick] (n2) to (n4);
\draw[->, thick] (n2) to (n5);
\draw[->, thick] (n2) to  (n6);
\draw[->, thick] (n2) to (n7);

\draw[->>, thick] (n3) to [out=115,in=325]  (n1);
\draw[->>, thick] (n3) to [bend right] (n2);
\draw[->>, thick] (n2) to [bend left] (n3);
\draw[->>, thick] (n3) to  (n4);
\draw[->>, thick] (n3) to (n5);
\draw[->>, thick] (n3) to  (n6);
\draw[->>, thick] (n3) to (n7);

\draw[->, dashed] (n4) to [out=115,in=265] (n1);
\draw[->, dashed] (n4) to [out=90,in=245] (n2);
\draw[->, dashed] (n4) to [bend right] (n3);
\draw[->, dashed] (n5) to  (n4);
\draw[->, dashed] (n1) to [out=245,in=90] (n5);
\draw[->, dashed] (n4) to  (n6);
\draw[->, dashed] (n4) to (n7);

\draw[->>, dashed] (n5) to [out=90,in=245] (n1);
\draw[->>, dashed] (n5) to [out=60,in=220] (n2);
\draw[->>, dashed] (n5) to [out=40,in=195] (n3);
\draw[->>, dashed] (n5) to [bend right] (n4);
\draw[->>, dashed] (n4) to [bend left] (n5);
\draw[->>, dashed] (n5) to (n6);
\draw[->>, dashed] (n5) to (n7);

\draw[->, densely dotted] (n6) to [out=65,in=225]  (n1);
\draw[->, densely dotted] (n6) to [out=40,in=195] (n2);
\draw[->, densely dotted] (n6) to [out=10,in=170] (n3);
\draw[->, densely dotted] (n6) to [out=330,in=160]  (n4);
\draw[->, densely dotted] (n6) to [bend right] (n5);
\draw[->, densely dotted] (n7) to  (n6);
\draw[->, densely dotted] (n1) to [bend right] (n7);

\draw[->>, densely dotted] (n7) to [bend left] (n1);
\draw[->>, densely dotted] (n7) to [out=10,in=170] (n2);
\draw[->>, densely dotted] (n7) to [out=350,in=145] (n3);
\draw[->>, densely dotted] (n7) to [out=300,in=140] (n4);
\draw[->>, densely dotted] (n7) to [out=280,in=120](n5);
\draw[->>, densely dotted] (n7) to [bend right] (n6);
\draw[->>, densely dotted] (n6) to [bend left] (n7);
\end{tikzpicture}
\end{subfigure}
\caption{Construction of networks $Q_{0,k}$ for $k=1,2,3$, showing how additional cells and edge types are added to support 3D synchrony subspaces.}
\label{fig:pn1}
\end{figure}
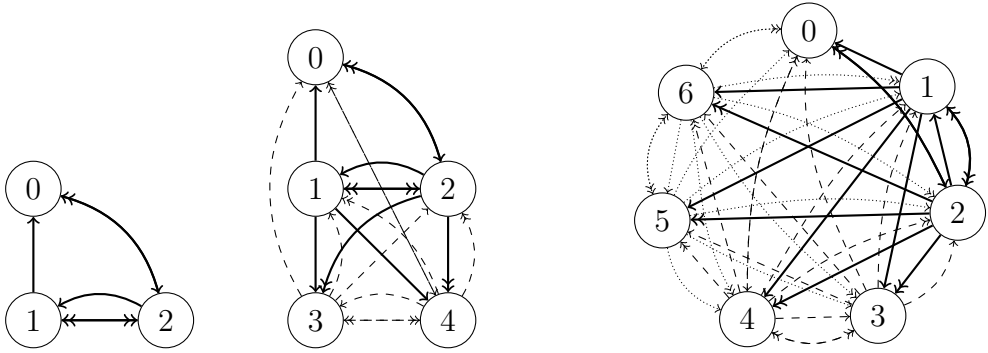

If there exists a node with two or less outgoing edges, $n_1>0$, we start with the network $P_{n_1}$ and inductively add two new cells in each step.
Define $Q_{n_1,0}:=P_{n_1}$ the coupled cell network defined in Section~\ref{sec:realbookemb} with $n_1+1$ cells and $n_1$ edge types. 
Assume that the cells of $Q_{n_1,k}$ are $\{0,1,\dots, n_1+2k\}$ and denote by $E_j$ the set of couplings with type $i_j$, where $j=1,\dots, n_1+2k$.
The network $Q_{n_1,k+1}$ is the network with cells $\{0,1,\dots, n_1+2k,n_1+2k+1,n_1+2(k+1)\}$ and two new edges types $i_{n_1+2k+1}$ and $i_{n_1+2(k+1)}$.
For each $j=1,\dots,n_1+2k$, edges with type $i_{j}$ are $E_j\cup\{(j,n_1+2k+1),(j,n_1+2(k+1))\}$.
The edges of type $i_{n_1+2k+1}$ are 
$$\{(n_1+2k+1,c):c=0,\dots, n_1+2k\}\cup\{( n_1+2(k+1),n_1+2k+1),(0 ,n_1+2(k+1))\}.$$
And the edges of type $i_{n_1+2(k+1)}$ are 
$$\{(n_1+2(k+1),c):c=0,\dots, n_1+2k\}\cup\{(n_1+2(k+1)),n_1+2k+1),(n_1+2k+1,n_1+2(k+1))\}.$$
See Figure~\ref{fig:pn2} for some examples.

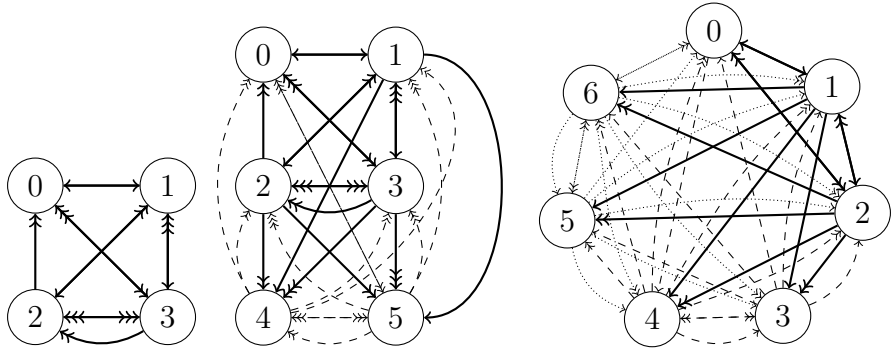
\begin{figure}[h]
\centering
%\begin{subfigure}{.2\textwidth}
\begin{tikzpicture}
\node (n1) [circle,draw]   {0};
\node (n2) [circle,draw] [right=of n1]  {1};
\node (n3) [circle,draw] [below=of n1]  {2};
\node (n4) [circle,draw] [below=of n2]  {3};

\draw[->, thick] (n2) to (n1);
\draw[->, thick] (n1) to (n2);
\draw[->, thick] (n2) to (n3);
\draw[->, thick] (n2) to (n4);

\draw[->>, thick] (n3) to (n1);
\draw[->>, thick] (n3) to (n2);
\draw[->>, thick] (n4) to [bend left] (n3);
\draw[->>, thick] (n1) to (n4);

\draw[->>>, thick] (n4) to (n1);
\draw[->>>, thick] (n4) to (n2);
\draw[->>>, thick] (n4) to (n3);
\draw[->>>, thick] (n3) to (n4);
\end{tikzpicture}
%\end{subfigure}
%\begin{subfigure}{.35\textwidth}
\begin{tikzpicture}
\node (n1) [circle,draw]   {0};
\node (n2) [circle,draw] [right=of n1]  {1};
\node (n3) [circle,draw] [below=of n1]  {2};
\node (n4) [circle,draw] [below=of n2]  {3};
\node (n5) [circle,draw] [below=of n3]  {4};
\node (n6) [circle,draw] [below=of n4]  {5};

\draw[->, thick] (n2) to (n1);
\draw[->, thick] (n1) to (n2);
\draw[->, thick] (n2) to (n3);
\draw[->, thick] (n2) to (n4);
\draw[->, thick] (n2) to (n5);
\draw[->, thick] (n2) to [out=0, in=0] (n6);

\draw[->>, thick] (n3) to (n1);
\draw[->>, thick] (n3) to (n2);
\draw[->>, thick] (n4) to [bend left] (n3);
\draw[->>, thick] (n1) to (n4);
\draw[->>, thick] (n3) to (n5);
\draw[->>, thick] (n3) to (n6);

\draw[->>>, thick] (n4) to (n1);
\draw[->>>, thick] (n4) to (n2);
\draw[->>>, thick] (n4) to (n3);
\draw[->>>, thick] (n3) to (n4);
\draw[->>>, thick] (n4) to (n5);
\draw[->>>, thick] (n4) to (n6);

\draw[->>, dashed] (n5) to [bend left] (n1);
\draw[->>, dashed] (n5) to [out=10, in=330] (n2);
\draw[->>, dashed] (n5) to [bend left] (n3);
\draw[->>, dashed] (n5) to [bend right] (n4);
\draw[->>, dashed] (n6) to [bend left] (n5);
\draw[->>, dashed] (n1) to (n6);

\draw[->>>, dashed] (n6) to  (n1);
\draw[->>>, dashed] (n6) to [bend right] (n2);
\draw[->>>, dashed] (n6) to [bend left] (n3);
\draw[->>>, dashed] (n6) to [bend right] (n4);
\draw[->>>, dashed] (n6) to (n5);
\draw[->>>, dashed] (n5) to (n6);
\end{tikzpicture}
\begin{tikzpicture}
\node (n1) at ({2*sin(0)},{2*cos(0)}) [circle,draw]   {0};
\node (n2)  at ({2*sin(51)},{2*cos(51)}) [circle,draw]  {1};
\node (n3) [circle,draw] at ({2*sin(102)},{2*cos(102)})  {2};
\node (n4) [circle,draw] at ({2*sin(153)},{2*cos(153)})  {3};
\node (n5) [circle,draw] at ({2*sin(204)},{2*cos(204)})  {4};
\node (n6) [circle,draw] at ({2*sin(255)},{2*cos(255)})  {5};
\node (n7) [circle,draw] at ({2*sin(306)},{2*cos(306)})  {6};

\draw[->, thick] (n2) to (n1);
\draw[->, thick] (n1) to (n2);
\draw[->, thick] (n2) to (n3);
\draw[->, thick] (n2) to (n4);
\draw[->, thick] (n2) to (n5);
\draw[->, thick] (n2) to (n6);
\draw[->, thick] (n2) to (n7);

\draw[->>, thick] (n3) to (n1);
\draw[->>, thick] (n3) to  (n2);
\draw[->>, thick] (n1) to  (n3);
\draw[->>, thick] (n3) to  (n4);
\draw[->>, thick] (n3) to  (n5);
\draw[->>, thick] (n3) to  (n6);
\draw[->>, thick] (n3) to  (n7);

\draw[->, dashed] (n4) to  (n1);
\draw[->, dashed] (n4) to [out=90, in=245] (n2);
\draw[->, dashed] (n4) to [bend right] (n3);
\draw[->, dashed] (n5) to [bend right] (n4);
\draw[->, dashed] (n1) to (n5);
\draw[->, dashed] (n4) to  (n6);
\draw[->, dashed] (n4) to  (n7);

\draw[->>, dashed] (n5) to [out=90, in=250] (n1);
\draw[->>, dashed] (n5) to [out=60, in=220] (n2);
\draw[->>, dashed] (n5) to [out=20, in=220] (n3);
\draw[->>, dashed] (n5) to  (n4);
\draw[->>, dashed] (n4) to  (n5);
\draw[->>, dashed] (n5) to  (n6);
\draw[->>, dashed] (n5) to  (n7);

\draw[->, densely dotted] (n6) to  (n1);
\draw[->, densely dotted] (n6) to [out=40,in=195] (n2);
\draw[->, densely dotted] (n6) to [out=10,in=170] (n3);
\draw[->, densely dotted] (n6) to [out=330,in=160] (n4);
\draw[->, densely dotted] (n6) to [bend right] (n5);
\draw[->, densely dotted] (n7) to [bend right] (n6);
\draw[->, densely dotted] (n1) to (n7);

\draw[->>, densely dotted] (n7) to  (n1);
\draw[->>, densely dotted] (n7) to [out=10,in=170] (n2);
\draw[->>, densely dotted] (n7) to  [out=350,in=145] (n3);
\draw[->>, densely dotted] (n7) to  [out=300,in=140] (n4);
\draw[->>, densely dotted] (n7) to [out=280,in=120] (n5);
\draw[->>, densely dotted] (n7) to  (n6);
\draw[->>, densely dotted] (n6) to  (n7);
\end{tikzpicture}
%\end{subfigure}
\caption{Networks $Q_{1,1}$, $Q_{1,2}$ and $Q_{2,2}$ that combine 2D and 3D synchrony subspaces to support almost complete realizations of heteroclinic networks.}
\label{fig:pn2}
\end{figure}

In the next result, we describe the minimal synchrony subspaces of the networks $Q_{n_1,n_2}$ and the eigenvalues of the Jacobian matrix $J_f^{Q_{n_1,n_2}}$.
Denote by $\Delta_j$ the 2D space given by $\{(x_0,\dots,x_k): x_1=x_i, i\neq j \}$ and by $\Delta_{j_1,j_2}$ the 3D space given by $\{(x_0,\dots,x_k): x_1=x_i, i\neq j_1,j_2 \}$.

\begin{prop}
The networks $Q_{n_1,n_2}$ has $n_1+2n_2+1$ cells and $n_1+2n_2+1$ edge types..
Any coupled cell system associated with $Q_{n_1,n_2}$ admits the following minimal synchrony subspaces: 
$$\Delta_1,\dots, \Delta_{n_1}, \Delta_{n_1+1,n_1+2},\dots, \Delta_{n_1+2n_2-1,n_1+2n_2}.$$
\end{prop}

\begin{proof}
The subspaces are invariant since the form of a coupled cell system $f^{Q_{n_1,n_2}}$ is
$$\begin{cases}
\dot{x}_0							=	f(x_0,	x_{1},x_2,\dots,x_{n_1}						,x_{n_1+1},x_{n_1+2},\dots,x_{n_1+2 n_2-1},x_{n_1+2 n_2})\\
\dot{x}_{1}						=f(x_{1},	x_{0},x_2,\dots,x_{n_1}						,x_{n_1+1},x_{n_1+2},\dots,x_{n_1+2 n_2-1},x_{n_1+2 n_2})\\
\vdots\\
\dot{x}_{n_1}					=f(x_{n_1},	x_{1},x_{2},\dots,x_{0}					,x_{n_1+1},x_{n_1+2},\dots,x_{n_1+2 n_2-1},x_{n_1+2 n_2})\\
\dot{x}_{n_1+1}				=f(x_{n_1+1},	x_{1},x_2,\dots,x_{n_1}				,x_{n_1+2},x_{n_1+2}			,\dots,x_{n_1+2 n_2-1},x_{n_1+2 n_2})\\
\dot{x}_{n_1+2}				=f(x_{n_1+2},	x_{1},x_2,\dots,x_{n_1}				,x_{0},x_{n_1+1}			,\dots,x_{n_1+2 n_2-1},x_{n_1+2 n_2})\\
\vdots\\
\dot{x}_{n_1+2 n_2-1}	=f(x_{n_1+2 n_2-1},x_{1},x_2,\dots,x_{n_1}	,x_{n_1+1},x_{n_1+2},\dots,x_{n_1+2 n_2},x_{n_1+2 n_2})\\
\dot{x}_{n_1+2 n_2}		=f(x_{n_1+2 n_2}	,x_{1},x_2,\dots,x_{n_1}	,x_{n_1+1},x_{n_1+2},\dots,x_{0},x_{n_1+2 n_2-1})
\end{cases}.$$
Moreover, the $3D$ synchrony subspaces $\Delta_{n_1+1,n_1+2}$,..., $\Delta_{n_1+2n_2-1,n_1+2n_2}$ are minimal since the subspaces $\Delta_{k}$ where $k=n_1+1,\dots,n_1+2n_2$ are not invariant.
\end{proof}

{\bf Local dynamics near equilibrium points}

The equilibrium points of the heteroclinic network will be placed in the 1D full synchrony subspace $\Delta_0\subset\mathbb{R}^{1+n_1+2n_2}$.
We check that there exists a function $f:\mathbb{R}^{1+n_1+2n_2}\rightarrow \mathbb{R}$ such that each equilibrium point has the unstable manifold contained in exactly one of the previous synchrony subspaces.

\begin{prop}
Let $p\in\Delta_0$ be a equilibrium of a coupled cell system $f^{Q_{n_1,n_2}}$ for some  $f:\mathbb{R}^{1+n_1+2n_2}\rightarrow \mathbb{R}$.
Denote by $J_f^{Q_{n_1,n_2}}$ the Jacobian matrix at that point and by $f_j$ the derivative of $f(y_0,y_1,\dots,y_{1+n_1+2n_2})$ with respect to $y_j$.
The eigenvalues of $J_f^{Q_{n_1,n_2}}$ restricted to $\Delta_j$, $j=1,\dots,n_1$, are:
$$f_0(p)+f_1(p)+\dots+f_{n_1+2n_2}(p),\quad f_0(p)-f_j(p).$$
And the eigenvalues of $J_f^{Q_{n_1,n_2}}$ restricted to $\Delta_{n_1+2j-1,n_1+2j}$, $j=1,\dots,n_2$, are:
$$f_0(p)+f_1(p)+\dots+f_{n_1+2n_2}(p),$$
$$\dfrac{1}{2}(2f_0(p)-f_{n_1+2j-1}(p)-f_{n_1+2j}(p)\pm\sqrt{f_{n_1+2j}(p)^2+2f_{n_1+2j-1}(p)f_{n_1+2j}(p)-3f_{n_1+2j-1}(p)^2}).$$

\end{prop}

\begin{proof}
The eigenvalues of $J_f^{Q_{n_1,n_2}}$ in the subspaces $\Delta_j$, $j=1,\dots, n_1$, are analogous to the ones obtain in Section~\ref{sec:realbookemb}.
The Jacobian matrix of $f^{Q_{n_1,n_2}}$ restricted to $\Delta_{n_1+2k-1,n_1+2k}$ at $p$ has the form:
$$\begin{bmatrix}
   f_0(p)+b & f_{2j}(p) & f_{2j+1}(p)\\
   b   & f_0(p) & f_{2j}(p)+f_{2j+1}(p)\\
   b+f_{2j}(p) & f_{2j+1}(p)  &  f_0(p)\\
\end{bmatrix},
$$
where $b=\displaystyle\sum_{\substack{i=1\\i\neq 2j+1,2j+2}}^{2k+1} f_i(p)$. Computing the eigenvalues, we obtain the result.
\end{proof}

%Next, we check that we can select the derivatives of $f$ at the equilibrium points to the coupled cell system be locally compatible with the intended realization of the heteroclinic network.
Order the equilibrium nodes of the heteroclinic network $\mathcal{N}$ such that the node with two or less outgoing connections are the first $n_1$. 
For each equilibrium node $v_i$ of the heteroclinic network take a full-synchronous point  $p_i\in\Delta_0$.
The partial derivative of $f$ with respect to $y_l$ at $p_i$ will be given by a set of constant $\alpha_l^i$ to be selected for $i=1,\dots,n_1+n_2$ and $l=0,\dots,n_1+2n_2$.

For $i=1,\dots,n_1$, we want $p_i$ to be a source in $\Delta_i\setminus \Delta_0$ and to be a sink in the other synchrony subspaces $\Delta_j$ and $\Delta_{n_1+2k-1,n_1+2k}$ for $j=1,\dots,i-1,i+1,n_1$and $k=1,\dots,n_2$.
This means that $\alpha_0^i+\alpha_1^i+\dots+\alpha_{n_1+2n_2}^i<0$, $\alpha_0^i-\alpha_j^i<0$, $2\alpha_0^i-\alpha_{n_1+2k-1}^i-\alpha_{n_1+2k}^i\pm\sqrt{{\alpha_{n_1+2k}^i}^2+2\alpha_{n_1+2k-1}^i\alpha_{n_1+2k}^i-3{\alpha_{n_1+2k-1}^i}^2}<0$ and $\alpha_0^i-\alpha_i^i>0$, for $j\ne i$ and $k=1,\dots,n_2$.
In order to see that the previous condition can be satisfied take for example $\alpha_0^i=-1$, $\alpha_j^i=0$ for $j=1,\dots,n_1$ and $j\neq i$, $\alpha_{n_1+2k}=\alpha_{n_1+2k-1}^i=0$ for $k=1,\dots,n_2$, and $\alpha_i^i=-2$.

And for $i=n_1+1,\dots,n_1+n_2$, we want $p_i$ to be a source in $\Delta_{n_1+2(i-n_1)-1,n_1+2(i-n_1)}\setminus \Delta_0$ and to be a sink on the other $\Delta_j$ and $\Delta_{n_1+2k,n_1+2k+1}$ for  $j=1,\dots,n_1$ and $k\neq i-n_1$.
In order to achieve this the constants $\alpha_l^i$ need to satisfy the following inequalities: $\alpha_0^i+\alpha_1^i+\dots+\alpha_{n_1+2n_2}^i<0$, $\alpha_0^i-\alpha_j^i<0$, 
$$2\alpha_0^i-\alpha_{n_1+2k-1}^i-\alpha_{n_1+2k}^i\pm\sqrt{{\alpha_{n_1+2k}^i}^2+2\alpha_{n_1+2k-1}^i\alpha_{n_1+2k}^i-3{\alpha_{n_1+2k-1}^i}^2}<0$$ and 
$$2\alpha_0^i-\alpha_{n_1+2(i-n_1)-1}^i-\alpha_{n_1+2(i-n_1)}^i\pm\sqrt{{\alpha_{n_1+2(i-n_1)}^i}^2+2\alpha_{n_1+2(i-n_1)-1}^i\alpha_{n_1+2(i-n_1)}^i-3{\alpha_{n_1+2(i-n_1)-1}^i}^2}>0,$$ for $j=1,\dots, n_1$ and $k=1,\dots,n_2$ such that $k\neq i-n_1$.  
A solutions to the previous inequalities is given by $ \alpha_0^i=-1$, $\alpha_j^i=0$, $\alpha_{n_1+2k-1}^i= \alpha_{n_1+2k}^i=0$ and $\alpha_{n_1+2(i-n_1)-1}^i= \alpha_{n_1+2(i-n_1)}^i=-2$ for $j=1,\dots,n_1$ and $k=1,\dots,n_2$ such that $k\neq i-n_1$.

Now, we start to build the function $f$ which defines the coupled cell system which realizes the heteroclinic network. 
As we did in Section~\ref{sec:realbookemb}, first we define $f$ in a small neighborhood of the full-synchrony equilibrium nodes $p_i$.
We consider neighborhoods given by solid cylinders around the diagonal and denote by $C_{\epsilon}(p_i)$ the solid cylinder with radius $\epsilon$ and height $2\epsilon$ such that the center axis of the cylinder coincide with the diagonal points and the point $p_i$ is in the middle of the solid cylinder.       
Taking constants $\alpha_l^i\in\mathbb{R}$ satisfying the local conditions stated above and $\epsilon>0$ small enough, define $f:\mathbb{R}^{n_1+2n_2+2}\rightarrow \mathbb{R}$ as follows:
$$f(y_0,y_1,\dots,y_{n_1+2n_2+1})=\sum_{i=1}^{n_1+n_2}\delta_{p_i}(y_0,y_1,\dots,y_{n_1+2n_2+1}) \sum_{l=0}^{n_1+2n_2+1} \alpha_l^{p_i}(y_l-\rho_i),$$
where $p_i=(\rho_i,\dots,\rho_i)\in\Delta_0$, $\delta_{p_i}$ is a bump function which is $1$ if $(x_0,x_1,\dots,x_{n_1+2n_2+1})\in C_{\epsilon}((\rho_i,\dots,\rho_i))$ and it is $0$ if $(x_0,x_1,\dots,x_{n_1+2n_2+1})\notin C_{2\epsilon}((\rho_i,\dots,\rho_i))$.

{\bf Realizations of heteroclinic connections}

The last step to realize the heteroclinic network is to create the heteroclinic connections in the appropriate synchrony subspaces. 
Analogous to the book embedding realization, we embed the heteroclinic connections.
However, we consider a distinct embedding for the outgoing connections from each node of $\mathcal{N}$ instead of all together.
For the $n_1$ nodes with two or less outgoing connections, the outgoing connections are embedded in the plane $\mathbb{R}^2$ as they will be realized in a $2D$ synchronous subspace, $\Delta_j$. 
And the outgoing connections of the other $n_2$ nodes are embedding in the space $\mathbb{R}^3$ since they will be realized in a $3D$ synchronous subspace, $\Delta_{j_1,j_2}$.
As the realization of the heteroclinic connections leave the local dynamics, it will transit to the horizontal flow given by the embedding. 

In order to deal with the transitions, for the first $n_1$ equilibrium nodes, we consider a rectangle in $\mathbb{R}^2$ such that the two sides are parallel to the diagonal $x=y$ and two sides are perpendicular to the diagonal $x=y$.
Moreover, the vertices of the rectangle are in the circle center in $(\rho_i,\rho_i)$ and radius $2\epsilon_i$.
Denote by $R_1,R_2,\dots,R_{n_1}$ the corresponding rectangles around the points $(\rho_1,\rho_1),(\rho_2,\rho_2),\dots,(\rho_{n_1},\rho_{n_1})$.
For the other $n_2$ equilibrium nodes, we consider a prism in $\mathbb{R}^3$ with a $k$-sided polygon base where $k$ is the number of outgoing connection such that the base is orthogonal to the diagonal $x=y=z$ and the edges of the prism parallel to the diagonal.
We also impose the vertices of the prism to be on the cylinder surface around the diagonal with radius $2\epsilon_i$.
Denote by $S_{n_1+1},S_{n_1+2},\dots,S_{n_1+n_2}$ the prisms around the points $(\rho_{n_1+1},\rho_{n_1+1},\rho_{n_1+1}),(\rho_{n_1+2},\rho_{n_1+2},\rho_{n_1+2}),\dots,(\rho_{n_1+n_2},\rho_{n_1+n_2},\rho_{n_1+n_2})$.
See Figures~\ref{fig:2dsquare} and \ref{fig:3dprism} for an illustration.
Note that, we can assume that the correspondent rectangle in $\Delta_j$ or prism on $\Delta_{n_1+2j,n_1+2j+1}$ intersect the unstable manifold on the faces parallel to the diagonal by shrinking the distance to the diagonal. 
So the heteroclinic connections will cross the sides of these rectangles and prism transversely.

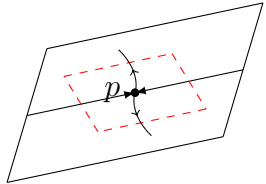
\begin{figure}[h]
\centering
\tdplotsetmaincoords{115}{70}
\begin{tikzpicture} [tdplot_main_coords,scale=0.5]
\tdplotsetrotatedcoords{0}{50}{0}

\draw[tdplot_rotated_coords] (0,0,2)  to  (4,4,6);
\draw[tdplot_rotated_coords] (4,4,6)  to (4,4,2);
\draw[tdplot_rotated_coords] (4,4,2)  to (0,0,-2);
\draw[tdplot_rotated_coords] (0,0,-2)  to  (0,0,2);

\node (P) [tdplot_rotated_coords,circle,draw,fill=black, inner sep=0pt, outer sep =0pt, minimum size=1mm, label={left}:$p$] at (2,2,2) {};

\draw[tdplot_rotated_coords,-latex] (0,0,0) to (2,2,2);
\draw[tdplot_rotated_coords,-latex] (4,4,4) to (2,2,2);

\draw[tdplot_rotated_coords,dashed, red] (0.5,0.5,1.5)  to  (2.5,2.5,3.5);
\draw[tdplot_rotated_coords,dashed, red] (2.5,2.5,3.5)  to (3.5,3.5,2.5);
\draw[tdplot_rotated_coords,dashed, red] (3.5,3.5,2.5)  to (1.5,1.5,0.5);
\draw[tdplot_rotated_coords,dashed, red] (1.5,1.5,0.5)  to  (0.5,0.5,1.5);

\foreach \t in {0,5,...,60}
\draw[tdplot_rotated_coords] ({1 + cos(\t)},{1+ cos(\t)},{2 + sin(\t)}) -- ({1 + cos(\t + 5)},{1+ cos(\t + 5)},{2 + sin(\t + 5)});

\draw[tdplot_rotated_coords,->] ({1 + cos(30)},{1+ cos(30)},{2 + sin(30)}) -- ({1+ cos(30 + 5)},{1+ cos(30 + 5)},{2 + sin(30 + 5)});

\foreach \t in {0,5,...,60}
\draw[tdplot_rotated_coords] ({3 - cos(\t)},{3- cos(\t)},{2 - sin(\t)}) -- ({3 - cos(\t + 5)},{3- cos(\t + 5)},{2 - sin(\t + 5)});

\draw[tdplot_rotated_coords,->] ({3 - cos(30)},{3- cos(30)},{2 - sin(30)}) -- ({3 - cos(30 + 5)},{3- cos(30 + 5)},{2 - sin(30 + 5)});

%\draw[-latex,  tdplot_rotated_coords,dashed] (3,0,6) to  (1.5,0,6);
%\draw[-latex,  tdplot_rotated_coords] (1.5,0,6) to  (0,0,6);
\end{tikzpicture}
\caption{A rectangular region in a 2D synchrony subspace surrounding an equilibrium point. The embedded heteroclinic connection crosses the rectangle transversely, facilitating the transition from local dynamics to the global trajectory.}
\label{fig:2dsquare}
\end{figure}

\begin{figure}[h]
\centering
\tdplotsetmaincoords{115}{70}
\begin{tikzpicture} [tdplot_main_coords,scale=0.5]
\tdplotsetrotatedcoords{0}{50}{0}

%\draw[tdplot_rotated_coords] (0,0,2)  to  (4,4,6);
%\draw[tdplot_rotated_coords] (4,4,6)  to (4,4,2);
%\draw[tdplot_rotated_coords] (4,4,2)  to (0,0,-2);
%\draw[tdplot_rotated_coords] (0,0,-2)  to  (0,0,2);

\node (P) [tdplot_rotated_coords,circle,draw,fill=black, inner sep=0pt, outer sep =0pt, minimum size=1mm, label={left}:$p$] at (2,2,2) {};
\draw[tdplot_rotated_coords,-latex] (-2,-2,-2) to (2,2,2);
\draw[tdplot_rotated_coords,-latex] (6,6,6) to (2,2,2);

\filldraw [tdplot_rotated_coords,dashed,draw=red,fill=red,opacity=0.3] 
        ({1+cos(0)},{1+sin(0)},{1-cos(0)-sin(0)}) -- ({1+cos(72)},{1+sin(72)},{1-cos(72)-sin(72)}) -- ({1+cos(144)},{1+sin(144)},{1-cos(144)-sin(144)}) -- ({1+cos(216)},{1+sin(216)},{1-cos(216)-sin(216)}) -- ({1+cos(288)},{1+sin(288)},{1-cos(288)-sin(288)}) -- ({1+cos(360)},{1+sin(360)},{1-cos(360)-sin(360)}) -- cycle ;

\draw[tdplot_rotated_coords,dashed, red] ({1+cos(0)},{1+sin(0)},{1-cos(0)-sin(0)})  to  ({1+cos(72)},{1+sin(72)},{1-cos(72)-sin(72)});
\draw[tdplot_rotated_coords,dashed, red] ({1+cos(72)},{1+sin(72)},{1-cos(72)-sin(72)})  to  ({1+cos(144)},{1+sin(144)},{1-cos(144)-sin(144)});
\draw[tdplot_rotated_coords,dashed, red] ({1+cos(144)},{1+sin(144)},{1-cos(144)-sin(144)})  to  ({1+cos(216)},{1+sin(216)},{1-cos(216)-sin(216)});
\draw[tdplot_rotated_coords,dashed, red] ({1+cos(216)},{1+sin(216)},{1-cos(216)-sin(216)})  to  ({1+cos(288)},{1+sin(288)},{1-cos(288)-sin(288)});
\draw[tdplot_rotated_coords,dashed, red] ({1+cos(288)},{1+sin(288)},{1-cos(288)-sin(288)})  to  ({1+cos(360)},{1+sin(360)},{1-cos(360)-sin(360)});

\filldraw [tdplot_rotated_coords,dashed,draw=red,fill=red,opacity=0.3] ({3+cos(0)},{3+sin(0)},{3-cos(0)-sin(0)})  -- ({3+cos(72)},{3+sin(72)},{3-cos(72)-sin(72)})  		-- ({3+cos(144)},{3+sin(144)},{3-cos(144)-sin(144)}) -- ({3+cos(216)},{3+sin(216)},{3-cos(216)-sin(216)}) -- ({3+cos(288)},{3+sin(288)},{3-cos(288)-sin(288)}) -- cycle;

\draw[tdplot_rotated_coords,dashed, red] ({3+cos(0)},{3+sin(0)},{3-cos(0)-sin(0)})  to  ({3+cos(72)},{3+sin(72)},{3-cos(72)-sin(72)});
\draw[tdplot_rotated_coords,dashed, red] ({3+cos(72)},{3+sin(72)},{3-cos(72)-sin(72)})  to  ({3+cos(144)},{3+sin(144)},{3-cos(144)-sin(144)});
\draw[tdplot_rotated_coords,dashed, red] ({3+cos(144)},{3+sin(144)},{3-cos(144)-sin(144)})  to  ({3+cos(216)},{3+sin(216)},{3-cos(216)-sin(216)});
\draw[tdplot_rotated_coords,dashed, red] ({3+cos(216)},{3+sin(216)},{3-cos(216)-sin(216)})  to  ({3+cos(288)},{3+sin(288)},{3-cos(288)-sin(288)});
\draw[tdplot_rotated_coords,dashed, red] ({3+cos(288)},{3+sin(288)},{3-cos(288)-sin(288)})  to  ({3+cos(360)},{3+sin(360)},{3-cos(360)-sin(360)});

\filldraw [tdplot_rotated_coords,dashed,draw=red,fill=red,opacity=0.2] ({1+cos(0)},{1+sin(0)},{1-cos(0)-sin(0)}) -- ({1+cos(72)},{1+sin(72)},{1-cos(72)-sin(72)}) -- ({3+cos(72)},{3+sin(72)},{3-cos(72)-sin(72)}) --    ({3+cos(0)},{3+sin(0)},{3-cos(0)-sin(0)})-- cycle;

\filldraw [tdplot_rotated_coords,dashed,draw=red,fill=red,opacity=0.2] ({1+cos(72)},{1+sin(72)},{1-cos(72)-sin(72)}) -- ({1+cos(144)},{1+sin(144)},{1-cos(144)-sin(144)}) -- ({3+cos(144)},{3+sin(144)},{3-cos(144)-sin(144)}) --   ({3+cos(72)},{3+sin(72)},{3-cos(72)-sin(72)}) -- cycle;

\filldraw [tdplot_rotated_coords,dashed,draw=red,fill=red,opacity=0.2] ({1+cos(144)},{1+sin(144)},{1-cos(144)-sin(144)}) -- ({1+cos(216)},{1+sin(216)},{1-cos(216)-sin(216)}) -- ({3+cos(216)},{3+sin(216)},{3-cos(216)-sin(216)}) --   ({3+cos(144)},{3+sin(144)},{3-cos(144)-sin(144)}) -- cycle;

\filldraw [tdplot_rotated_coords,dashed,draw=red,fill=red,opacity=0.2] ({1+cos(216)},{1+sin(216)},{1-cos(216)-sin(216)}) -- ({1+cos(288)},{1+sin(288)},{1-cos(288)-sin(288)}) -- ({3+cos(288)},{3+sin(288)},{3-cos(288)-sin(288)}) -- ({3+cos(216)},{3+sin(216)},{3-cos(216)-sin(216)}) -- cycle;

\filldraw [tdplot_rotated_coords,dashed,draw=red,fill=red,opacity=0.2]  ({1+cos(288)},{1+sin(288)},{1-cos(288)-sin(288)}) -- ({1+cos(360)},{1+sin(360)},{1-cos(360)-sin(360)}) -- ({3+cos(0)},{3+sin(0)},{3-cos(0)-sin(0)}) -- ({3+cos(288)},{3+sin(288)},{3-cos(288)-sin(288)})  -- cycle;

\draw[tdplot_rotated_coords,dashed, red] ({1+cos(0)},{1+sin(0)},{1-cos(0)-sin(0)})  				to   ({3+cos(0)},  {3+sin(0)},  {3-cos(0)-sin(0)})  			;
\draw[tdplot_rotated_coords,dashed, red] ({1+cos(72)},{1+sin(72)},{1-cos(72)-sin(72)})  		to   ({3+cos(72)}, {3+sin(72)}, {3-cos(72)-sin(72)})  	;
\draw[tdplot_rotated_coords,dashed, red] ({1+cos(144)},{1+sin(144)},{1-cos(144)-sin(144)})  to   ({3+cos(144)},{3+sin(144)},{3-cos(144)-sin(144)});
\draw[tdplot_rotated_coords,dashed, red] ({1+cos(216)},{1+sin(216)},{1-cos(216)-sin(216)})  to   ({3+cos(216)},{3+sin(216)},{3-cos(216)-sin(216)});
\draw[tdplot_rotated_coords,dashed, red] ({1+cos(288)},{1+sin(288)},{1-cos(288)-sin(288)})  to   ({3+cos(288)},{3+sin(288)},{3-cos(288)-sin(288)});

%\foreach \t in {0,0.01,...,1}
%\draw[tdplot_rotated_coords] ({2 + \t*cos(36+\t)},{2+ \t*sin(36+\t)},{2 - \t*sin(36+\t) - \t*cos(36+\t)}) -- ({2 + \t*cos(36+\t+5)},{2+ \t*sin(36+\t+5)},{2 - \t*sin(36+\t+5) - \t*cos(36+\t+5)}) ;

\draw[tdplot_rotated_coords] (2,2,2) to ({2 + 1.5*cos(36)},    {2+ 1.5*sin(36)},    {2 - 1.5*sin(36)     - 1.5*cos(36)});
\draw[tdplot_rotated_coords] (2,2,2) to ({2 + 4*cos(120)}, {2+ 4*sin(120)}, {2 - 4*sin(120)  - 4*cos(120)});
\draw[tdplot_rotated_coords] (2,2,2) to ({2 + 2*cos(36+144)},{2+ 2*sin(36+144)},{2 - 2*sin(36+144) - 2*cos(36+144)});
\draw[tdplot_rotated_coords] (2,2,2) to ({2 + 2*cos(36+216)},{2+ 2*sin(36+216)},{2 - 2*sin(36+216) - 2*cos(36+216)});
\draw[tdplot_rotated_coords] (2,2,2) to ({2 + 4*cos(36+288)},{2+ 4*sin(36+288)},{2 - 4*sin(36+288) - 4*cos(36+288)});

\draw[tdplot_rotated_coords,->] (2,2,2) to ({2 + 0.8*cos(36)},    {2+ 0.8*sin(36)},    {2 - 0.8*sin(36)     - 0.8*cos(36)});
\draw[tdplot_rotated_coords,->] (2,2,2) to ({2 + cos(120)}, {2+ sin(120)}, {2 - sin(120)  - cos(120)});
\draw[tdplot_rotated_coords,->] (2,2,2) to ({2 + cos(36+144)},{2+ sin(36+144)},{2 - sin(36+144) - cos(36+144)});
\draw[tdplot_rotated_coords,->] (2,2,2) to ({2 + cos(36+216)},{2+ sin(36+216)},{2 - sin(36+216) - cos(36+216)});
\draw[tdplot_rotated_coords,->] (2,2,2) to ({2 + 2*cos(36+288)},{2+ 2*sin(36+288)},{2 - 2*sin(36+288) - 2*cos(36+288)});

%  to   ({3+cos(288)},{3+sin(288)},{3-cos(288)-sin(288)});

%
%\draw[tdplot_rotated_coords,->] ({1 + cos(30)},{1+ cos(30)},{2 + sin(30)}) -- ({1+ cos(30 + 5)},{1+ cos(30 + 5)},{2 + sin(30 + 5)});
%
%
%\foreach \t in {0,5,...,60}
%\draw[tdplot_rotated_coords] ({3 - cos(\t)},{3- cos(\t)},{2 - sin(\t)}) -- ({3 - cos(\t + 5)},{3- cos(\t + 5)},{2 - sin(\t + 5)});
%
%\draw[tdplot_rotated_coords,->] ({3 - cos(30)},{3- cos(30)},{2 - sin(30)}) -- ({3 - cos(30 + 5)},{3- cos(30 + 5)},{2 - sin(30 + 5)});

%\draw[-latex,  tdplot_rotated_coords,dashed] (3,0,6) to  (1.5,0,6);
%\draw[-latex,  tdplot_rotated_coords] (1.5,0,6) to  (0,0,6);
\end{tikzpicture}
\caption{A prism with a 5-sided polygonal base embedded in a 3D synchrony subspace.
The prism surrounds an equilibrium point, and each lateral face supports a distinct outgoing heteroclinic connection, enabling multidirectional transitions.}
\label{fig:3dprism}
\end{figure}
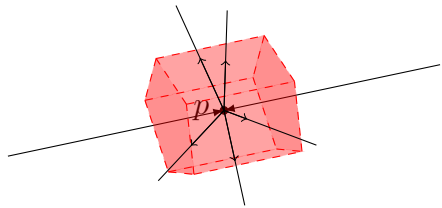

Take an equilibrium node $n_i$, $i=1,\dots,n_1$, of the heteroclinic network $\mathcal{N}$ with two or less outgoing connection. 
If $n_i$ has only one outgoing connection targeting $n_t$, we embed that connection in the plane twice: one in the half plane $x>y$ and another in the other half plane $x<y$.
Let $\gamma^i_1,\gamma^i_2:[0,1]\rightarrow \mathbb{R}^2$ be the these embedding, i.e., $\gamma^i_1(0)=\gamma^i_1(0)=(\rho_i,\rho_i)$, $\gamma^i_1(1)=\gamma^i_1(1)=(\rho_t,\rho_t)$,  $\gamma^i_1(]0,1[)\subset\{x>y\}$ and $\gamma^i_2(]0,1[)\subset\{x<y\}$.
If $n_i$ has exactly two outgoing connection targeting $n_{t_1}$ and $n_{t_2}$, we embed these connections in different half planes.
Denote these embeddings as $\gamma^i_1,\gamma^i_2:[0,1]\rightarrow \mathbb{R}^2$ such that $\gamma^i_1(0)=\gamma^i_1(0)=(\rho_i,\rho_i)$, $\gamma^i_1(1)=(\rho_{t_1},\rho_{t_1})$ and $\gamma^i_1(1)=(\rho_{t_2},\rho_{t_2})$, $\gamma^i_1(]0,1[)\subset\{x>y\}$ and $\gamma^i_2(]0,1[)\subset\{x<y\}$. 
In the first case, we also use the notation $t_1$ and $t_2$ as $t_1=t_2=t$.
For every $j=1,2$ and $i=1,\dots,n_1$, we assume, by changing the embedding, that $\gamma^i_j$ transversely crosses exactly once the rectangle $R_i$ along an edge parallel to the diagonal.
This means that there  exists an unique $\tau_R$ such that $\gamma^i_j(\tau_R)\in R_i$.
Moreover, we assume that  $\gamma^i_j$ is sufficiently far from the diagonal except when at the start and end point.
So there are $\tau_S$ and $\tau_T$ such that $\gamma^i_j([0,\tau_S[)\subset B_{2\epsilon}(\rho_{i},\rho_{i})$, $\gamma^i_j(]\tau_T,1])\subset B_{2\epsilon}(\rho_{t_j},\rho_{t_j})$ and $\gamma^i_j([\tau_S,\tau_T])\cap B_{2\epsilon}(\rho,\rho)=\emptyset$ for any $\rho\in\mathbb{R}$.

Given a node $n_i$, $i=n_1+1,\dots, n_1+n_2$, with $k$ outgoing connections targeting the nodes $n_{t_1},n_{t_2},\dots, n_{t_k}$, where $k>2$.
We embed these $k$ outgoing connections in $\mathbb{R}^3$ away from the planes $x=y$, $y=z$ and $x=z$. 
Denote by $\gamma^i_1,\gamma^i_2,\dots, \gamma_k^i:[0,1]\rightarrow \mathbb{R}^3$ embeddings of the $k$ outgoing connections such that
$\gamma^i_j(0)=(\rho_i,\rho_i,\rho_i)$, $\gamma^i_j(1)=(\rho_{t_j},\rho_{t_j},\rho_{t_j})$ and $\gamma^i_j(]0,1[)\cap \{x=y \vee y=z \vee x=z\}=\emptyset$, for $j=1,\dots,k$.
Again, we assume that each embedding transversely crosses a distinct lateral faces of the $k$-prism $S_{i}$ around $(\rho_i,\rho_i,\rho_i)$ and it is sufficiently far from the diagonal, except close to  the end points.
This also means that there are unique $\tau_R$, $\tau_S$ and $\tau_T$ such that $\gamma^i_j(\tau_R)\in S_i$, $\gamma^i_j([0,\tau_S[)\subset B_{2\epsilon}(\rho_{i},\rho_{i},\rho_{i})$, $\gamma^i_j(]\tau_T,1])\subset B_{2\epsilon}(\rho_{t_j},\rho_{t_j},\rho_{t_j})$ and $\gamma^i_j([\tau_S,\tau_T])\cap B_{2\epsilon}(\rho,\rho,\rho)=\emptyset$ for any $\rho\in\mathbb{R}$, $i=n_1+1,\dots, n_1+n_2$ and $j=1,\dots,k$.

We make the following assumptions on the embedding, except at the start and target point:
\begin{itemize}
\item Any two embedding outgoing from the same node do not intersect.
This means that $\gamma_{j_1}^i(]0,1[)\cap \gamma_{j_2}^i(]0,1[)=\emptyset$ for any $i=1,\dots, n_1+n_2$.

\item Embedding of outgoing connections from two of the first $n_1$ equilibrium node do not intersect or their intersection respects the conditions presented in Figure~\ref{fig:adj2dtraj}.

\item Embedding of outgoing connections from different nodes with more than two outgoing connections also do not intersect, as the embedding is in $\mathbb{R}^3$.
For any $i_1,i_2=n_1+1,\dots,n_1+n_2$, this means that $\gamma_{j_1}^{i_1}(]0,1[)\cap \gamma_{j_2}^{i_2}(]0,1[)=\emptyset$. 
Thus, there are tubular neighborhood around the paths that do not intersect.
\end{itemize}

Now, we complement the function $f$ away from the diagonal using the previous embedding.
For the first $n_1$ nodes with two or less outgoing connection, we repeat the process done in Section~\ref{sec:realbookemb}.
Take one of these node $n_i$, $i=1,\dots,n_1$, and consider the embeddings $\gamma_1^i$ and $\gamma_2^i$ in $\mathbb{R}^2$ of the heteroclinic connections.
These heteroclinic connections will be realized inside the synchrony subspace   $\Delta_i$.
We can do this node by node, as we did in Section~\ref{sec:realbookemb}, since the intersection between the embedding of two outgoing connections is empty or it respect the conditions presented in Figure~\ref{fig:adj2dtraj}.
So, we obtain a function $f$ such that the coupled cell system $\dot{x}=f^{Q_{n_1,n_2}}(x)$ realizes the outgoing connections from the first $n_1$ nodes.
Note that the unstable manifold at the equilibrium points $p_i=(\rho_i,\rho_i,\dots, \rho_{i})\in\mathbb{R}^{n_1+2n_2+1}$ is one dimensional.
So the two embeddings $\gamma_1^i$ and $\gamma_2^i$ force the unstable manifold to be fully contained in the heteroclinic network.

Last, we focus on the equilibrium nodes than have more than two outgoing heteroclinic connections.
Let $n_i$ be a node with $k>2$ outgoing connections and $\gamma^i_j$ be the embedding of the outgoing connection in $\mathbb{R}^3$, where $i=n_1+1,\dots, n_2$ and $j=1,\dots,k$.
These heteroclinic connections will be realized inside the synchrony subspace $\Delta_{n_1+2(i-n_1)-1,n_1+2(i-n_1)}=\Delta_{2i-n_1-1,2i-n_1}$.  
The coupled cell systems inside this synchrony subspace has the form:
$$\begin{cases}
\dot{x}_0					=f(x_0					,x_0,\dots,x_{0},x_{2i-n_1-1}	,x_{2i-n_1}		,x_{0},\dots,x_{0})\\
\dot{x}_{2i-n_1-1}=f(x_{2i-n_1-1}	,x_0,\dots,x_{0},x_{2i-n_1}		,x_{2i-n_1}		,x_{0},\dots,x_{0})\\
\dot{x}_{2i-n_1}	=f(x_{2i-n_1}		,x_0,\dots,x_{0},x_{0}				,x_{2i-n_1-1}	,x_{0},\dots,x_{0})
\end{cases}.$$

For a given embedding $\gamma^i_j=(\gamma_1,\gamma_2,\gamma_3)$ targeting the node $n_t$, there exist times $\tau_S$ and $\tau_T$ where $\gamma^i_j(\tau_S)$ crosses the prism $S_i$ around $(\rho_i,\rho_i,\rho_i)$ and $\gamma^i_j(\tau_T)$ crosses a ball centered in $(\rho_t,\rho_t,\rho_t)$ with radius $\epsilon>0$.
The heteroclinic connection will follow the arc $\psi:[\tau_S,\tau_T]\rightarrow \mathbb{R}^{n_1+2n_2+1}$ given by $\psi_i(t)=\gamma_1(t)$, $i\neq 2i-n_1-1, 2i-n_1$, $\psi_{2i-n_1-1}(t)=\gamma_2(t)$ and $\psi_{2i-n_1}(t)=\gamma_3(t)$.
Since the synchrony subspace $\Delta_{2i-n_1-1,2i-n_1}$ is three dimensional, we need to find three tubular neighborhoods corresponding to the input of $x_0,x_{2i-n_1-1},x_{2i-n_1}$ where the function $f$ will be changed.
Let $A$, $B$ and $C$ be the tubular neighborhoods around $\psi([\tau_S,\tau_T])$, $$\{(\gamma_2(t),  \gamma_1(t),\dots,\gamma_1(t),\gamma_3(t),\gamma_3(t),\gamma_1(t),\dots,\gamma_1(t)): t\in [\tau_S,\tau_T]\},$$ and $$\{(\gamma_3(t),\gamma_1(t),\dots,\gamma_1(t),\gamma_1(t),\gamma_2(t)	,\gamma_1(t),\dots,\gamma_1(t)): t\in [\tau_S,\tau_T]\}.$$ 
These tubular neighborhoods are disjoint since the embedding $\gamma^i_j$ does not intersect $\{x=y\vee y=z\vee x=z\}$.
Denote by $\mathfrak{f}$ the face of the prism $S_i$ which is crossed by $\gamma^i_j$, we assume that the base of $A$, $B$ and $C$ is equal to the respective lift of face $\mathfrak{f}$ to $\mathbb{R}^{n_1+2n_2+1}$.
So the base of $A$ is equal to $\{(x_0,x_1,\dots, x_{n_1+2n_2}):  \forall_{(x,y,z)\in \mathfrak{f}} x_i=x, x_{2i-n_1-1}=y,x_{2i-n_1-1}=z \}$.
And the other cases are anologous.
Now, we are in conditions to add the relevant terms to the function $f$ in order to realize the heteroclinic connection.
We add the following term to the function $f(y_0,y_1,\dots,y_{n_1+2n_2})$:
$$\delta_{A}(y_0,y_1,\dots,y_{n_1+2n_2})\dot{\gamma_1}(t)+\delta_{B}(y_0,y_1,\dots,y_{n_1+2n_2})\dot{\gamma_2}(t)+\delta_{C}(y_0,y_1,\dots,y_{n_1+2n_2+1})\dot{\gamma_3}(t),$$
where $\delta_{A},\delta_{B},\delta_{C}$ are bump functions that are zero outside $A,B,C$ and greater than zero inside $A,B,C$, 
and $t$ is the time such that the projection of $(y_0,y_1,\dots,y_{n_1+2n_2+1})$ into the center of $A$ corresponds to the time $t$ when $(y_0,y_1,\dots,y_{n_1+2n_2+1})\in A$ for $i=1,2,3$.

The unstable manifold of $(\rho_i,\dots, \rho_i)$ is two dimensional, it is contained in $\Delta_{2i-n_1-1,2i-n_1}$ and it intersects the lateral faces of the prism $S_i$.
In particular, the intersection of $W^u(p_v)$ with the lateral face $\mathfrak{f}$ has dimension $1$, i.e. it is a line segment.
Considering the trajectory that passes in any point of this line segment, as we go backwards we tend to the equilibrium node $(\rho_i,\dots, \rho_i)$. 
And as we go forwards, we follow the trajectory of $\psi$ until we reach the stable manifold of $(\rho_t,\dots, \rho_t)$.
Then, we continue to the equilibrium point $(\rho_t,\dots, \rho_t)$ which is a sink in $\Delta_{2i-n_1-1,2i-n_1}$.
Thus the heteroclinic connection from $n_i$ to $n_t$ is realized in the coupled cell system $\dot{x}=f^{Q_{n_1,n_2}}(x)$.
The heteroclinic connections lie in the minimal synchrony subspace $\Delta_{2i-n_1-1,2i-n_1}$ where the target equilibrium nodes are stable,
thus the heteroclinic connections are robust to small perturbations of the function $f$.
We can repeat the previous process for the other outgoing connections from this node, since we assume that theirs embedding does not overlap, except at the starting and targeting nodes.
The intersection of the unstable manifold $W^u((\rho_i,\dots, \rho_i))$ with the lift of the prism $S_i$ is inside the synchrony subspace $\Delta_{2i-n_1-1,2i-n_1}$ and it is given by a $k$-polygon.
The trajectories passing through this polygon, except its vertices, converge to one of the targeting equilibrium point.
The set of trajectories passing through the vertices of the polygon has measure zero.
So the unstable manifold $W^u((\rho_i,\dots, \rho_i))$ is contained in the heteroclinic network, except for a set of zero measure.

Note that the previous terms do not destroy the outgoing heteroclinic connections from the first $n_1$ equilibrium nodes as the embedding in $\mathbb{R}^3$ do not intersect $\{x=y\vee y=z\vee x=z\}$.
Thus, the heteroclinic connections previous realized in $f^{Q_{n_1,n_2}}$ are not affected and continue to be realized. 
We can repeat the previous process for the other equilibrium nodes with more than two outgoing connections, because the embeddings in $\mathbb{R}^3$ do not intersect each other.
This finishes the proof that the coupled cell system $\dot{x}=f^{Q_{n_1,n_2}}(x)$ realize the heteroclinic network $\mathcal{N}$ in a robust and almost complete way.

%Since the book embedding of connections do not intersect in $\mathbb{R}^3$, the tubular neighbourhoods $\delta_{X_1},\delta_{X_2},\delta_{X_3}$ for the different connections can be selected to avoid intersections. 
%Thus we can repeat the previous steps for the other nodes with three or more outgoing edges and find a function $f$ such that every connection of $\mathcal{N}$ are realized in the coupled cell system $f^{Q_{n_1,n_2}}$. 
%Moreover, this heteroclinic realization is robust for small perturbation of the function $f$ and it is also almost complete. 
%This finishes the proof of Theorem~\ref{teo:almostcomplete}. 
%The final function $f$ that we obtain has the following form:
%$$f(x_0,x_1,\dots,x_{n_1+2n_2+1})=\sum_{i=1}^{n_1+n_2}\delta_{p_i}(x_0,x_1,\dots,x_{n_1+2n_2+1}) \sum_{l=0}^{n_1+2n_2+1} \alpha_l^{p_i}(x_l-\rho_i)$$
%$$ + \sum_{i=1}^{n_1} \sum_{e\in \mathcal{N}_i} \delta_{A^{e,i}}(x_0,x_1,\dots,x_{n_1+2n_2+1})\dot{(\phi_i(e))_1}(t)+\delta_{B^{e,i}}(x_0,x_1,\dots,x_{n_1+2n_2+1})\dot{(\phi_i(e))_2}(t).$$
%$$+\sum_{i=n_1+1}^{n_2} \sum_{e\in \mathcal{N}_i} \delta_{X_1^{e,i}}(x_0,x_1,\dots,x_{n_1+2n_2+1})\dot{(\xi_i(e))_1}(t)+\delta_{X_2^{e,i}}(x_0,x_1,\dots,x_{n_1+2n_2+1})\dot{(\xi_i(e))_2}(t)$$
%$$+\delta_{X^{e,i}_3}(x_0,x_1,\dots,x_{n_1+2n_2+1})\dot{(\xi_i(e))_3}(t).$$
%

\section{Conclusions and Future work}
This work presents new methods for realizing heteroclinic networks in coupled cell systems placing the heteroclinic connection on both 2D and 3D synchrony subspaces. 
By adapting Field's construction for 2D synchrony subspaces and book embedding concept, we showed that the number of required cells can be minimized based on the network's book-thickness.
We also extended the framework to allow almost complete realizations as the heteroclinic connection belong to 3D synchrony subspaces.
The present work leads to questions about the stability of the constructed heteroclinic networks and the nature of switching between equilibria.

\section*{Acknowledgments}
PS was supported by Project ISEG Research - UID/06522/2025 financed by FCT/MCTES through national funds.

% ---------------------------------------------

\end{document}